\documentclass{elsarticle}

\usepackage[utf8]{inputenc}
\usepackage{amsmath, amssymb, amsthm, amsfonts}
\usepackage{bm}
\usepackage{mathtools}
\usepackage{stmaryrd}
\usepackage{caption, subcaption}
\usepackage{hyperref} 
\usepackage{cleveref}
\usepackage[titletoc]{appendix}
\usepackage[normalem]{ulem}
\usepackage{geometry}
\usepackage{algorithm}
\usepackage{algpseudocode}
\usepackage{listings}
\usepackage{tikz, graphicx}

\usepackage{multirow}

\usetikzlibrary{positioning}

\theoremstyle{definition}

\theoremstyle{plain}
\newtheorem{lemma}{Lemma}
\theoremstyle{plain}

\newtheorem{assumption}{Assumption}

\renewcommand{\hat}[1]{\widehat{#1}} 
\newcommand{\diff}[1]{\mathop{}\!{\mathrm{d}#1}} 
 
\newcommand{\pd}[2]{\frac{\partial#1}{\partial#2}}

\newcommand{\nor}[1]{\left\| #1 \right\|} 
\newcommand{\LRp}[1]{\left( #1 \right)} 
\newcommand{\LRs}[1]{\left[ #1 \right]} 
\newcommand{\LRa}[1]{\left\langle #1 \right\rangle} 
\newcommand{\LRb}[1]{\left| #1 \right|} 
\newcommand{\LRc}[1]{\left\{ #1 \right\}} 
\newcommand{\jump}[1] {\ensuremath{\left\llbracket#1\right\rrbracket}} 
\newcommand{\avg}[1] {\ensuremath{\LRc{\!\!\LRc{#1}\!\!}}}

\DeclareMathOperator{\sgn}{sgn}
\newcommand{\ProjV}{\Pi_N\bm{v}(\bm{u}_h)}

\newcommand{\refone}[1]{{\color{red}#1}}
\newcommand{\reftwo}[1]{{\color{blue}#1}}

\renewcommand{\refone}[1]{#1}
\renewcommand{\reftwo}[1]{#1}

\newtheorem{remark}{Remark}

\begin{document}

\begin{frontmatter}
\title{On the choice of viscous discontinuous Galerkin discretization for entropy correction artificial viscosity methods}

\author[rice]{Samuel Van Fleet\texorpdfstring{\corref{cor1}\fnref{presaddr}}{}}
\ead{vanflees@uww.edu}

\author[Oden,ASE]{Jesse Chan}
\ead{jesse.chan@oden.utexas.edu}

\cortext[cor1]{Corresponding author.}
\fntext[presaddr]{Present address: Department of Mathematics, University of Wisconsin--Whitewater, Whitewater, WI, USA}

\address[rice]{Department of Computational Applied Mathematics and Operations Research, Rice University}
\address[Oden]{Oden Institute for Computational Engineering and Sciences, The University of Texas at Austin}
\address[ASE]{Department of Aerospace Engineering \& Engineering Mechanics, The University of Texas at Austin}

\date{}

\begin{abstract}
Entropy correction artificial viscosity (ECAV) is an approach for enforcing a semi-discrete entropy inequality through an entropy dissipative correction term. The resulting method can be implemented as an artificial viscosity with an extremely small viscosity coefficient. In this work, we analyze ECAV when the artificial viscosity is discretized using a local discontinuous Galerkin (LDG) method. We prove an $O(h)$ upper bound on the ECAV coefficient, indicating that ECAV does not result in a restrictive time-step condition. We additionally show that ECAV is contact preserving, and compare ECAV to traditional shock capturing artificial viscosity methods. 
\end{abstract}
\end{frontmatter}


\section{Introduction} 

High order discontinuous Galerkin (DG) methods are often used to approximate solutions to time-dependent nonlinear conservation laws. High order approximations are efficient and accurate due to their small dissipation and dispersion error relative to low order methods. However, in the presence of under-resolved solutions such as shocks, high order methods can become unstable. 

Artificial viscosity is another popular technique for stabilization; an incomplete list of examples is \cite{berthon2023artificial, glaubitz2019smooth, guermond2011entropy, hennemann2021provably, lv2016entropy, klockner2011viscous, majda1979numerical, persson2006sub, vonneumann1950method} and the references therein. Many artificial viscosity approaches utilize shock capturing to detect regions where the solution is under-resolved and add artificial viscosity to suppress Gibbs-type oscillations. One drawback is that these methods often rely on parameters that need to be tuned for different problems. 

Entropy stable DG formulations were introduced in \cite{carpenter2014entropy, chen2017entropy, gassner2016split} to stabilize DG schemes without introducing heuristic stabilization parameters. The resulting methods are nodal DG formulations which use collocated Gauss-Lobatto quadrature nodes, resulting in mass and stiffness matrices which satisfy the summation-by-parts properties \cite{gassner2013skew}. In addition, these methods rely on a ``flux differencing'' framework, which reformulates a standard nodal DG method to resemble an algebraic finite volume formulation with central fluxes. To enforce entropy stability, these central fluxes are replaced with entropy conservative finite volume fluxes \cite{chandrashekar2013kinetic,  ismail2009affordable, ranocha2018comparison, tadmor1987numerical}. 
If a consistent two-point entropy stable numerical flux is used at the interfaces, the method satisfies a semi-discrete cell entropy inequality. These entropy stabilization techniques have also been extended to non-collocation formulations \cite{chan2018discretely, CICCHINO2025}. 

Flux differencing entropy stable DG formulations have been successfully applied to several challenging problems involving the compressible Navier-Stokes equations in \cite{PARSANI2021}, magnetohydrodynamics (MHD) equations in \cite{BOHM2020, RUEDARAMIREZ2021}, the shallow water equations \cite{wintermeyer2017entropy, wintermeyer2018entropy}, and hypersonics \cite{OBLAPENKO2025, PEYVAN2023}. However, they require the computation of an entropy conservative numerical flux, which can be expensive to evaluate and complex to derive.
Moreover, while the extension of such methods to more general discretizations is possible \cite{chan2018discretely, CICCHINO2025, keim2025}, the resulting entropy stable flux differencing formulation becomes more complicated.

Recently, entropy correction artificial viscosity (ECAV) methods were introduced in \cite{chan2025artificial,  christner2025entropystablefinitedifference, christner2025entropystablenodaldiscontinuous} as a simpler way of enforcing a cell entropy inequality. ECAV is closely related to the entropy correction terms from \cite{abgrall2018general, abgrall2022reinterpretation, edoh2024conservative, gaburro2023high, mantri2024fully}. In \cite{chan2025artificial}, artificial viscosity was determined based on the local violation of a cell entropy inequality. The artificial viscosity coefficients are locally computed over each element without parameters and were shown to be small for a number of numerical examples involving the compressible Euler equations. The method is high-order accurate, satisfies the global semi-discrete entropy inequality, and has a large maximum stable time-step size when used in conjunction with explicit-time stepping schemes.

In this paper, we analyze the ECAV method from \cite{chan2025artificial}, which uses a BR-1 type scheme to discretize the artificial viscosity term. BR-1 discretizations of a Laplacian result in spurious non-constant null space eigenmodes \cite{hesthaven2007nodal, sherwin20062d} which result in arbitrarily large artificial viscosity coefficients. We show that when the artificial viscosity term is discretized using a local discontinuous Galerkin (LDG) formulation \cite{cockburn2007analysis, cockburn1998local}, it is bounded from above by $O(h)$. The resulting ECAV method results in smaller magnitude artificial viscosity coefficients and preserves the less restrictive hyperbolic CFL condition.

The organization of this paper is as follows: \Cref{sec:section_2} reviews nonlinear conservation laws and different forms of an entropy inequality, \Cref{sec:section_3} reviews a standard DG weak formulation for nonlinear conservation laws, as well as the LDG discretization of the entropy correction artificial viscosity term. In \Cref{sec:section_4}, we derive upper bounds on the artificial viscosity coefficients under LDG. In \Cref{sec:section_5} we present numerical experiments that validate the robustness, high-order accuracy, entropy stability, and upper bounds of the artificial coefficients of the method. We compare the ECAV method with BR-1 \cite{chan2025artificial} and LDG viscous discretizations, as well as with heuristic shock capturing approaches from \cite{hennemann2021provably, persson2006sub, PGSH2020}.

\section{Nonlinear conservation laws and entropy inequalities}\label{sec:section_2}

In this paper, we focus on the numerical approximation of solutions to systems of nonlinear conservation laws in $d$ spatial dimensions:
\begin{equation}
 \frac{ \partial \bm{u}}{\partial t} + \sum_{m=1}^d  \frac{\partial \bm{f}_m(\bm{u})}{\partial x_{m}} = 0,
\label{eq:ncl}
\end{equation}
where $\bm{u}(\bm{x},t) \in \mathbb{R}^n$ and $\bm{f}_m : \mathbb{R}^n \to \mathbb{R}^n$ are the flux functions. We assume that \eqref{eq:ncl} admits one or more entropy inequalities of the form
\begin{equation}
\frac{\partial S(\bm{u}) }{\partial t}+ \sum_{m=1}^d  \frac{ \partial F_m(\bm{u})}{\partial x_{m}} \le 0,
\label{eq:entropy inequality}
\end{equation}
where $S(\bm{u})$ is a scalar convex entropy and $F_m(\bm{u})$ are the associated entropy fluxes. The entropy variables are defined by 
\[ 
\bm{v}(\bm{u}) = \frac{\partial S}{\partial \bm{u}},
\]
and for sufficiently regular solutions, one can derive the equality version of \eqref{eq:entropy inequality} by multiplying \eqref{eq:ncl} by the entropy variables $\bm{v}$, using the chain rule, and the following compatibility condition for the entropy and entropy fluxes
\begin{equation}\label{eqn:compatibility condition}
\refone{\frac{\partial \bm{f}_m}{\partial \bm{u}}^T \bm{v}} = \frac{\partial F_m}{\partial \bm{u}}, \quad m = 1,\dots,d.
\end{equation}
Furthermore, the entropy fluxes can be written as
\begin{equation}\label{eqn:entropy flux decomposition}
F_m(\bm{u}) = \bm{v}(\bm{u})^T \bm{f}_m(\bm{u}) - \psi_m(\bm{u}), 
\end{equation}
where $\psi_m(\bm{u})$ are the entropy potentials. Differentiating \eqref{eqn:entropy flux decomposition} and applying \eqref{eqn:compatibility condition} yields
\begin{equation}
\frac{\partial \psi_m}{\partial \bm{u}}(\bm{u}) = 
\frac{\partial \bm{v}}{\partial \bm{u}}(\bm{u})^T \bm{f}_m(\bm{u}).
\label{eq:dpsi_rule}
\end{equation} 
The decomposition of the entropy flux into entropy variables and a potential term is standard in the design of entropy-conservative and entropy-stable numerical schemes \reftwo{\cite{ godunov1961, harten1983symmetric, lax1973hyperbolic, leveque2002finite, tadmor2003entropy}}. Finally, the \reftwo{strict} convexity of $S(\bm{u})$ guarantees that the mapping between conservative variables $\bm{u}$ and entropy variables $\bm{v}$ is invertible, and that the system is symmetrizable \cite{dafermos2005hyperbolic, friedrichs1971systems, lukavcova2025note}.

A vanishing viscosity argument yields the inequality version of \eqref{eq:entropy inequality} for more general classes of solutions \cite{chen2017entropy, dafermos2005hyperbolic, godlewski2013numerical}. In this work, we utilize an integrated cell version of \eqref{eq:entropy inequality} \cite{jiang1994cell}. Consider a closed domain $D \subset \mathbb{R}^d$ with boundary $\partial D$. Integrating \eqref{eq:entropy inequality} over $D$ and applying the divergence theorem yields
\begin{equation}\label{eq:cell entropy inequality}
    \int_D \frac{\partial S(\bm{u})}{\partial t} + \int_{\partial D} \sum_{m=1}^d (\bm{v}^T\bm{f}_m(\bm{u}) - \psi_m(\bm{u}))n_m \leq 0.
\end{equation}
We will enforce this cell entropy inequality \eqref{eq:cell entropy inequality} by enforcing an intermediate identity. To derive this intermediate inequality, first multiply \eqref{eq:ncl} by the entropy variables and integrate over some domain $D$. Integrating the spatial derivative term by parts yields
\begin{equation}
\int_{D} \pd{S(\bm{u})}{t} + \sum_{m=1}^d \LRs{\int_{D} -\pd{\bm{v}(\bm{u})^T}{x_m} \bm{f}_m(\bm{u}) + \int_{\partial D} \bm{v}(\bm{u})^T\bm{f}_m(\bm{u}) n_m} = 0.
\label{eq:entropy_identity_step}
\end{equation}
Subtracting \eqref{eq:cell entropy inequality} from \eqref{eq:entropy_identity_step} then yields
\begin{equation}
\sum_{m=1}^d \LRp{\int_{D}-\pd{\bm{v}(\bm{u})}{x_m}^T\bm{f}_m(\bm{u}) + \int_{\partial D} \psi_m(\bm{u}) n_m} \geq 0.
\label{eq:cell_entropy_identity}
\end{equation}
We note that for differentiable $\bm{u}$ and $\bm{f}_m(\bm{u})$, \eqref{eq:cell_entropy_identity} should be an \textit{equality} due to \eqref{eq:dpsi_rule}
\begin{equation}
\sum_{m=1}^d\LRp{\int_{D}-\pd{\bm{v}(\bm{u})}{x_m}^T\bm{f}_m(\bm{u}) + \int_{\partial D} \psi_m(\bm{u}) n_m} = 0.
\label{eq:cell_entropy_equality}
\end{equation}

\section{Local discontinuous Galerkin (LDG) discretization of artificial viscosity}\label{sec:section_3}

\subsection{High order DG formulation}
We assume that the domain $\Omega \subset \mathbb{R}^d$ is triangulated by non-overlapping simplicial elements $D^k$, where each element $D^k$ is the image of a reference simplex $\hat{D}$ under some affine mapping $\phi^k: \hat{D} \rightarrow D^k$. Let $\bm{n}$ denote the outward normal vector $\bm{n} = \LRs{n_1,\dots,n_d}$ on each face of $D^k$.
Finally, let $(u,v)_{D^k}, \LRa{u,v}_{\partial D^k}$ denote the $L^2$ inner products on $D^k$ and the surface $\partial D^k$
\[
(u,v)_{D^k} \approx \int_{D^k} u(\bm{x})v(\bm{x}) \diff{x}, \qquad \LRa{u,v}_{\partial D^k} \approx \int_{\partial D^k} u(\bm{x})v(\bm{x}) \diff{x},
\]
where the approximate equality is due to the assumption that both volume and surface integrals in inner products are approximated using quadrature. \refone{We make the mild assumption that the volume and surface integrals are approximated such that integration by parts holds discretely; this can be done using summation-by-parts quadratures \cite{hicken2016multidimensional} or volume quadrature of degree $2N-1$ and surface quadratures of degree $2N$ \cite{chan2019skew}.}

We consider a standard DG formulation for an approximate solution $\bm{u}_h$ on a single element $D^k$ for \eqref{eq:ncl} \cite{hesthaven2007nodal, karniadakis2005spectral}:
\begin{equation}
\LRp{\pd{\bm{u}_h}{t}, \bm{w}}_{D^k} + \sum_{m=1}^d \LRp{-\bm{f}_m(\bm{u}_h), \pd{\bm{w}}{x_m}}_{D^k} + \LRa{\bm{f}_n^*, \bm{w}}_{\partial D^k} = \bm{0}, \qquad \bm{w} \in \LRs{P^N(D^k)}^n,
\label{eq:dg_form}
\end{equation}
where $P^N(D^k)$ denotes the space of total degree $N$ polynomials on $D^k$. 
This formulation is derived by multiplying \eqref{eq:ncl} by a test function $\bm{w} \in \LRs{P^N(D^k)}^n$, integrating by parts, and introducing a numerical flux $\bm{f}^*_n$ across each inter-element interface. All volume and surface integrals are approximated using some quadrature rule, which will be specified later.

\subsection{Semi-discrete entropy estimate}

First, we observe that for systems of nonlinear conservation laws, the entropy variables $\bm{v}(\bm{u}_h)$ can be non-polynomial and do not necessarily lie in the test space $\LRs{P^N(D^k)}^n$. Thus, to derive a semi-discrete entropy estimate for \eqref{eq:dg_form}, we must instead test with the $L^2$ projection of the entropy variables $\Pi_N \bm{v}(\bm{u}_h)$ \cite{chan2018discretely, chan2025artificial}. Here, $\Pi_N$ denotes the $L^2$ projection operator onto $P^N(D^k)$ such that for $f\in L^2(D^k)$
\[
\LRp{\Pi_N f, v}_{D^k} = \LRp{f, v}_{D^k}, \qquad \forall v \in P^N(D^k).
\]
We note that similar observations on the non-polynomial nature of the entropy variables have been made by a variety of different groups in the literature \cite{andrews2024high, chan2022entropyprojection, colombo2022entropy, gkanis2021new, williams2019analysis}. 

Since $\pd{\bm{u}_h}{t} \in \LRs{P^N\LRp{D^k}}^n$ for method of lines discretizations, testing the time derivative term in \eqref{eq:dg_form} with the projected entropy variables yields
\begin{equation}
\LRp{\pd{\bm{u}_h}{t}, \Pi_N\bm{v}(\bm{u}_h)}_{D^k} = \LRp{\pd{\bm{u}_h}{t}, \bm{v}(\bm{u}_h)}_{D^k} = \LRp{\pd{S(\bm{u}_h)}{t}, 1}_{D^k},
\label{eq:dSdt}
\end{equation}
where we have used the chain rule in time for the final step. Note that this equality still holds under inexact quadrature, so long as the $L^2$ projection operator $\Pi_N$ is defined using the same inexact quadrature rule. 

Considering \eqref{eq:dg_form} tested with the projected entropy variables and using \eqref{eq:dSdt}, we recover the semi-discrete local rate of change of entropy of $D^k$ 

\begin{equation}
\LRp{\pd{S(\bm{u}_h)}{t}, 1}_{D^k} + \sum_{m=1}^d \LRp{-\bm{f}_m(\bm{u}_h), \pd{\Pi_N\bm{v}(\bm{u}_h)}{x_m}}_{D^k} + \LRa{\bm{f}_n^*, \Pi_N\bm{v}(\bm{u}_h)}_{\partial D^k} = 0.
\label{eq:local_dSdt}
\end{equation}
Due to the presence of the projected entropy variables and the fact that the inner products are approximated using quadrature, integration by parts \refone{and} the chain rule cannot be used to reproduce the proof of the continuous entropy (in)equality at the semi-discrete level from the DG formulation  \eqref{eq:local_dSdt}. In \cite{chan2025artificial}, a Laplacian artificial viscosity term is added to the high order DG method and discretized using BR-1 to recover an entropy inequality. In this work, we discretize the artificial viscosity term with local DG \cite{cockburn2007analysis, cockburn1998local} instead, which allows us to derive an upper bound on the artificial viscosity coefficient that is not achieved under a BR-1 viscous discretization.

\subsection{An entropy dissipative viscous LDG formulation}

We consider adding the following Laplacian artificial viscosity term to \eqref{eq:ncl}
\begin{equation}\label{eq: monolithic viscous regularization}
    \pd{\bm{u}}{t} + \sum_{m=1}^d \pd{\bm{f}_m(\bm{u})}{x_m} = \sum_{\reftwo{i = 1}}^d \pd{}{x_i}\left(\epsilon_k (\bm{u})\pd{\bm{u}}{\reftwo{x_i}} \right),
\end{equation}
where $\epsilon_k(\bm{u}) \geq 0$ \refone{is an element-wise artificial viscosity coefficient on the element $D^k$}. Applying the chain rule, \eqref{eq: monolithic viscous regularization} can be written as:
\begin{equation}\label{eq: chain rule regularization}
    \pd{\bm{u}}{t} + \sum_{m=1}^d \pd{\bm{f}_m(\bm{u})}{x_m} = \sum_{i = 1}^d \pd{}{x_i}\left(\epsilon_k(\bm{u})\pd{\bm{u}}{\bm{v}}\pd{\bm{v}}{x_i} \right).
\end{equation}
\reftwo{Assuming strict convexity} of $S(\bm{u})$ \cite{lukavcova2025note}, the Jacobian matrix $\pd{\bm{u}}{\bm{v}}$ is symmetric and positive definite. We will analyze the slightly more general viscous regularization:
\begin{equation}\label{eq: general viscous regularization}
\pd{\bm{u}}{t} + \sum_{m=1}^d \pd{\bm{f}_m(\bm{u})}{x_m} = \sum_{i,j = 1}^d \pd{}{x_i}\left(\epsilon_k(\bm{u})\bm{K}_{ij}\pd{\bm{v}}{x_j} \right),
\end{equation} 
where $\bm{K}_{ij}$ denotes blocks of a symmetric and positive semi-definite matrix $\bm{K}$

$$
\bm{K} = 
\begin{bmatrix}
    \bm{K}_{11} & \cdots & \bm{K}_{1d} \\
    \vdots      & \ddots & \vdots\\
    \bm{K}_{d1} & \cdots & \bm{K}_{dd}
\end{bmatrix} = \bm{K}^T, \quad \bm{K} \succeq 0.
$$
For example, taking $\bm{K}_{ij} = \delta_{ij} \pd{\bm{u}}{\bm{v}}$ recovers \eqref{eq: chain rule regularization}. We discretize \eqref{eq: general viscous regularization} by adding $\bm{g}_{\text{visc}}$ to the DG formulation 
\begin{equation}\label{eq: viscous DG formulation}
\LRp{\pd{\bm{u}_h}{t}, \bm{w} }_{D^k} + \sum_{m=1}^d \LRp{-\bm{f}_m(\bm{u}_{h}), \pd{\bm{w}}{x_m}}_{D^k} + \LRa{ \bm{f^{*}}_n,\bm{w}}_{\partial D^k} = \LRp{\bm{g}_{\text{visc}}, \bm{w}}_{D^k},
\end{equation}
where the viscous terms $\bm{g}_{\text{visc}}$ are discretizations of \eqref{eq: general viscous regularization} using a local DG discretization. \reftwo{For $i = 1,...,d$, $g_{\text{visc}}$ is given by}

\begin{gather}
    \begin{aligned}
        \LRp{\bm{\Theta}_i, \bm{w}_{1,i}}_{D^k} &= \LRp{-\reftwo{\ProjV},\pd{\bm{w}_{1,i}}{x_i}}_{D^k} \\
        &\quad + \LRa{ \left(\avg{\reftwo{\ProjV}} - \frac{\beta}{2}\jump{\reftwo{\ProjV}}\right)n_i,\bm{w}_{1,i} }_{\partial D^k}, \quad \forall \bm{w}_{1,i} \in \left[P^N(D^k) \right]^n \refone{,}
    \end{aligned} \label{eq: LDG 1} \\[1.5ex]
    \label{eq: LDG 2}
    \left(\bm{\sigma}_i, \bm{w}_{2,i} \right)_{D^k} = \left(\sum_{j=1}^d \epsilon_k(\bm{u}_h)\bm{K}_{ij} \bm{\Theta}_j, \bm{w}_{2,i} \right)_{D^k},\quad \forall \bm{w}_{2,i} \in \left[P^N(D^k) \right]^n \refone{,}\\[1.5ex]
    \label{eq: LDG 3}
    \left(\bm{g}_{\text{visc}},\bm{w}_3 \right)_{D^k} = \sum_{i=1}^d \left[\left(-\bm{\sigma}_i,\pd{\bm{w}_3}{x_i} \right)_{D^k} + \left\langle\left(\avg{\bm{\sigma}_i} + \frac{\beta}{2} \jump{\bm{\sigma_i}}\right)n_i,\bm{w}_3 \right\rangle_{\partial D^k}\right], \quad \forall \bm{w}_{3} \in \left[P^N(D^k) \right]^n.
\end{gather}
Here, $\jump{\cdot}$ and $\avg{\cdot}$ denote the jump and average operations:
\[
\jump{u} = u^+ - u^-, \qquad  \avg{u} = \frac{1}{2}\LRp{u^+ + u^-},
\]
where $u^-$ denotes the interior value of $u$ on a face of $D^k$, and $u^+$ denotes the exterior (neighboring) value of $u$ across the same face. \reftwo{$\ProjV$ denotes the $L^2$ projection of the entropy variables and} $\bm{\Theta}_i$ are DG approximations of derivatives of the \reftwo{projected entropy variables with respect to the $i$th coordinate}. In \eqref{eq: LDG 2}, we compute the viscous fluxes $\bm{\sigma}_i$ as the $L^2$ projection of $\sum_{j=1}^d \reftwo{\epsilon_k(\bm{u}_h)}\bm{K}_{ij}\bm{\Theta}_j$ for $i = 1,\ldots, d$ onto the approximation space of each element. The viscous terms  $\bm{g}_{\rm visc}$ are the result of computing the divergence of the viscous fluxes in \eqref{eq: LDG 3}. 

Additionally, let $\beta$ be LDG parameters which are constant over each face. In this work, we define \refone{$\beta = \sgn{(\bm{v}_0 \cdot \bm{n})}$}, where $\bm{v}_0 \in \mathbb{R}^d$ is some fixed vector \cite{cockburn2007analysis}. Setting $\beta = 0$ recovers the BR-1 formulation. We note that LDG can also be equivalently defined in terms of binary ``switches''  \cite{peraire2008compact}; both definitions of $\beta$ fall under the general formulation in \cite{cockburn1998local} and yield the same energy estimates.

The following lemma characterizes the global entropy dissipation estimate satisfied by \eqref{eq: viscous DG formulation}. The proof is similar to proofs given in \cite{chan2025artificial, chan2022entropy}, and is included for completeness.
\begin{lemma}
\label{lemma:ldg_dissipation}
    Let $\bm{g}_{\text{visc}}$ be given by \eqref{eq: LDG 1}, \eqref{eq: LDG 2}, and \eqref{eq: LDG 3}. Then, for a periodic domain,
    \begin{equation*}
        \sum_k -\left(\bm{g}_{\text{visc}}, \Pi_N \bm{v}(\bm{u}_h) \right)_{D^k} = \sum_k\sum_{i,j = 1}^d \left(\epsilon_k(\bm{u}_h)\bm{K}_{ij}\bm{\Theta}_j,\bm{\Theta}_i \right)_{D^k} \geq 0.   \end{equation*}
        where $\Pi_N\bm{v}(\bm{u}_h)$ denotes the $L^2$ projection of the entropy variables.
\end{lemma}
\begin{proof}  
    Let $\bm{w}_{1,i} = \bm{\sigma}_i$, integrate the volume term in \eqref{eq: LDG 1} by parts, let $\bm{w}_{2,i} = \Theta_i$, $\bm{w}_3 = \reftwo{\ProjV}$, and sum \eqref{eq: LDG 1} and \eqref{eq: LDG 2} from $1$ to $d$ to get
    \begin{gather}
        \label{eq: LDG 1b}
        \sum_{i=1}^d\left(\bm{\Theta}_i, \bm{\sigma}_i \right)_{D^k} = \sum_{i=1}^d\left(\pd{\reftwo{\ProjV}}{x_i},\bm{\sigma}_i \right)_{D^k} + \left\langle \frac{1}{2}\left(\jump{\reftwo{\ProjV}}(1-\beta) \right)n_i,\bm{\sigma}_i \right\rangle_{\partial D^k}, \\
        \label{eq: LDG 2b}
        \sum_{i=1}^d\left(\bm{\sigma}_i, \bm{\Theta}_i \right)_{D^k} = \left(\sum_{i,j=1}^d \epsilon_k(\bm{u}_h)\bm{K}_{ij} \bm{\Theta}_j, \bm{\Theta}_i \right)_{D^k},\\
        \label{eq: LDG 3b}
        \left(\bm{g}_{\text{visc}},\reftwo{\ProjV} \right)_{D^k} = \sum_{i=1}^d \left[\left(-\bm{\sigma}_i,\pd{\reftwo{\ProjV}}{x_i} \right)_{D^k} + \left\langle\left(\avg{\bm{\sigma}_i} + \frac{\beta}{2} \jump{\bm{\sigma}_i}\right)n_i,\reftwo{\ProjV} \right\rangle_{\partial D^k}\right].
    \end{gather}
    Substituting \eqref{eq: LDG 1b} into \eqref{eq: LDG 3b} and summing over $D^k$ shows that $\sum_{k}\left(\bm{g}_{\text{visc}}, \reftwo{\ProjV}\right)_{D^k}$ is equal to
    \begin{align*}        
        &\sum_k \sum_{i,j = 1}^{\refone{d}} \left( -\bm{\Theta}_i,\bm{\sigma}_i\right)_{D^k} + \left\langle \left(\frac{1}{2}\jump{\Pi_N\bm{v}(\bm{u}_{h})}(1-\beta)\right)n_i, \bm{\sigma}_i \right\rangle_{\partial D^k}
        + \left\langle \left(\avg{\bm{\sigma}_i} + \frac{\beta}{2}\jump{\bm{\sigma}_i}\right)n_i, \reftwo{\ProjV} \right\rangle_{\partial D^k} \refone{.}
    \end{align*}
    Summing \eqref{eq: LDG 2b} over the elements and substituting into the equation above we have
    \begin{equation}\label{eq: LDG disc}
    \begin{aligned}
        &\sum_{k}\LRp{\bm{g}_{\text{visc}}, \reftwo{\ProjV}}_{D^k} = \sum_k \sum_{i,j = 1}^{\refone{d}}-\LRp{\epsilon_k(\bm{u}_h)\bm{K}_{ij} \bm{\Theta}_j ,\bm{\Theta}_i}_{D^k}\\
        &  + \left\langle \left(\frac{1}{2}\jump{\reftwo{\Pi_N\bm{v}(\bm{u}_{h})}}(1-\beta)\right)n_i, \bm{\sigma}_i \right\rangle_{\partial D^k}
        + \left\langle \left(\refone{\avg{\bm{\sigma}_i}} + \frac{\beta}{2}\jump{\bm{\sigma}_i}\right)n_i, \reftwo{\ProjV} \right\rangle_{\partial D^k}\refone{.}
    \end{aligned}
    \end{equation}
For periodic boundary conditions, all faces are ``interior" faces \refone{and are shared by} two elements. We split the contributions from each surface term and swap them between $D^k$ and the neighboring element $D^{k,+}$, and rearrange the boundary terms noting that $n_i$, $\jump{\reftwo{\ProjV}}$, and $\beta$ change sign between $D^k$ and $D^{k,+}$.
The first surface term from \eqref{eq: LDG disc} is
\begin{equation*}
\begin{aligned}
    \frac{1}{2}&\sum_k\left\langle \left(\jump{\reftwo{\Pi_N\bm{v}(\bm{u}_{h})}}(1-\beta)\right)n_i, \bm{\sigma}_i \right\rangle_{\partial D^k}  \\ 
    & = \frac{1}{2} \sum_k\left(\frac{1}{2}\left\langle \jump{\reftwo{\ProjV}}(1 - \beta)n_i,\bm{\sigma}_i \right\rangle_{\partial D^k} + \left\langle \frac{1}{2}\jump{\reftwo{\ProjV}}(1-\beta)n_i,\bm{\sigma}_i \right\rangle_{\partial D^{k,+}}\right)\\
    & = \frac{1}{2} \sum_{k} \left\langle \left(\avg{\bm{\sigma}_i} + \frac{\beta}{2}\jump{ \bm{\sigma}_i}\right)n_i,\jump{ \reftwo{\ProjV}}\right\rangle \refone{.}
\end{aligned}
\end{equation*}
The second surface term from \eqref{eq: LDG disc} is
\begin{equation*}
    \begin{aligned}
        & \sum_k \left\langle \left(\avg{\bm{\sigma}_i} + \frac{\beta}{2}\jump{\bm{\sigma}_i}\right)n_i, \reftwo{\ProjV} \right\rangle_{\partial D^k}\\
        &= \frac{1}{2} \sum_k \left(\left\langle \left(\avg{\bm{\sigma}_i} + \frac{\beta}{2}\jump{\bm{\sigma}_i}\right)n_i,\reftwo{\ProjV}\right\rangle_{\partial D^k} + \left\langle \left(\avg{\bm{\sigma}_i} + \frac{\beta}{2}\jump{\bm{\sigma}_i}\right)n_i,\reftwo{\ProjV}\right\rangle_{\partial D^{k,+}}\right)\\
        &= -\frac{1}{2} \sum_k \left\langle \left(\avg{\bm{\sigma}_i} + \frac{\beta}{2}\jump{\bm{\sigma}_i}\right)n_i,\refone{\jump{\reftwo{\ProjV}}}\right\rangle_{\partial D^k}.
    \end{aligned}
\end{equation*}
Thus, the two surface terms cancel and by the positive semi-definiteness of $\bm{K}_{ij}$, we have
\begin{equation}\label{eq: LDG gvisc bound}
-\sum_{k}\LRp{\bm{g}_{\text{visc}}, \reftwo{\ProjV}}_{D^k} = \sum_k \sum_{i,j = 1}^{\refone{d}}\LRp{\epsilon_k(\bm{u}_h)\bm{K}_{ij} \bm{\Theta}_j ,\bm{\Theta}_i}_{D^k} \reftwo{\geq} 0.
\end{equation}
\end{proof} 
\refone{We note that Lemma~\ref{lemma:ldg_dissipation} can also be extended to curved elements using a skew-symmetric version of \eqref{eq: LDG 1}, \eqref{eq: LDG 2}, \eqref{eq: LDG 3} \cite{chan2022entropy}.}
Note that the LDG dissipation estimate in Lemma~\ref{lemma:ldg_dissipation} is exactly the same as the dissipation estimate for BR-1. Thus, Lemma 2 in \cite{chan2025artificial} holds for both BR-1 and LDG, which we restate without proof below:
\begin{lemma}[Lemma 2 in \cite{chan2025artificial}]
\label{lemma: entropy decay}
    Let $\epsilon_{k}(\bm{u}_h)$ on $D^k$ satisfy
    \begin{equation}
        \sum_{i,j = 1}^d\left(\epsilon_k(\bm{u}_h)\bm{K}_{ij}\bm{\Theta}_{j},\bm{\Theta}_i \right)_{D^k} \geq -\min{\left(0,\delta_k(\bm{u}_h)\right)},
    \end{equation}
    where the volume entropy residual $\delta_k(\bm{u}_h)$ is
    \begin{equation}\label{eqn: semi-discrete volume entropy residual}
        \delta_k(\bm{u}_h) = \sum_{m = 1}^d\left[\LRp{-\bm{f}_m(\bm{u}_h),\pd{\Pi_N\bm{v}\reftwo{(\bm{u}_h)}}{x_m}}_{D_k} + \LRa{\psi_m(\tilde{\bm{u}})n_m,1}_{\partial D_k} \right]\reftwo{,}
    \end{equation}
    \reftwo{and $\bm{\tilde{u}}$ denotes the entropy projection $\tilde{\bm{u}} = \bm{u}(\Pi_N \bm{v}(\bm{u}_h))$. Additionally, assume the numerical flux $\bm{f}^*_n$ is evaluated in terms of the entropy projection, then} \eqref{eq: viscous DG formulation} satisfies the following global entropy inequality:
    \begin{equation}
        \sum_k\LRs{\LRp{\pd{S(\bm{u}_h)}{t},1}_{D^k} + \LRa{\LRp{\Pi_N\bm{v}(\bm{u}_h)}^T\bm{f}^{*}_n - \sum_{m=1}^d\psi_{m}(\tilde{\bm{u}})n_m,1}_{\partial D^k}} \leq 0.
    \end{equation}
\end{lemma}
We note that under nodal collocation of both volume and surface quadratures, the evaluation of the entropy projection $\tilde{\bm{u}}$ reduces to the evaluation of the polynomial DG solution $\bm{u}_h$ at interfaces. 

Finally, a simple expression for the ECAV coefficient can be given by assuming that $\epsilon_{k}(\bm{u}_h)$ is constant over each element:
\begin{equation}
\epsilon_{k}(\bm{u}_h) = \frac{-\refone{\min{\left(0, \delta_k(\bm{u}_h)\right)}}}{\sum_{i,j=1}^d \LRp{\bm{K}_{ij}\bm{\Theta}_j, \bm{\Theta}_i}_{D^k}}, 
\label{eq:ECAV_coeff}
\end{equation}
where we approximate the ratio above via 
\begin{equation}
\frac{a}{b} \approx \frac{ab}{b^2 + \delta}.
\label{eq:regularized_ratio}
\end{equation}
\reftwo{In all of our experiments in this work, we use double-precision computations and} unless otherwise specified, we arbitrarily take $\delta = 10^{-14}$ for the experiments in this work. \reftwo{For single-precision computations, we expect that $\delta$ can be chosen proportionally to single-precision machine epsilon (e.g., $\delta = 10^{-7}$ or $10^{-6}$).} 

\section{Properties of ECAV under LDG}\label{sec:section_4}

The main contributions of this paper are to provide an upper bound based on \eqref{eq:ECAV_coeff} and properties of the LDG discretization. The motivation for the use of the LDG discretization in \reftwo{~\eqref{eq: LDG 1}--\eqref{eq: LDG 3}} is the availability of estimates for the gradient which do not hold for BR-1. These estimates enable an upper bound on the ECAV coefficient. 

\subsection{Estimates for the LDG gradient}

Let $V^N(D^k)$ denote the local approximation space on the element $D^k$. On simplices, $V^N(D^k) = P^N(D^k)$, while on quadrilateral or hexahedral elements, $V^N(D^k) = Q^N(D^k)$. The LDG gradient of a scalar quantity $u$ can be defined locally (integrating \eqref{eq: LDG 1} by parts) as
\begin{equation}
\LRp{\bm{\theta}, \bm{w}}_{D^k} = \LRp{\nabla u, \bm{w}}_{D^k} + \LRa{\frac{1}{2}\jump{u}\LRp{1 - \beta}, \bm{w}\cdot \bm{n}}_{\partial D^k}\refone{,} \qquad \forall \bm{w} \in \LRs{V^{N}(D^k)}^d,
\label{eq:ldg_grad}
\end{equation}
where $\beta = \refone{\sgn{(\bm{v}_0 \cdot\bm{n})}}$ for some \reftwo{non-zero,} divergence-free $\bm{v}_0 \in \mathbb{R}^d$.

Let $a \lesssim b$ denote $a \leq C b$ where $C$ is independent of the mesh size (though it may depend on shape regularity and polynomial degree). Restricting for the moment to the case when $D^k$ is an affine simplex, we have the following lemma: 
\begin{lemma}
\label{lemma:ldg}
Let $D^k$ be a $d$-dimensional affine simplex and $\beta = 1$ on at least one face. Then,
\[
\nor{\bm{\theta}}_{L^2(D^k)} \gtrsim \nor{\nabla u}_{L^2(D^k)}.
\]
\end{lemma}
\begin{proof}
To show this bound, we choose $\bm{w}$ in \eqref{eq:ldg_grad} to be some projection \reftwo{$\bm{\tau}^* \in \LRs{P^N(D^k)}^d$}. 
We use the projection from Section 3.1 of \cite{cockburn2007analysis} 
\begin{align}
\LRp{\bm{\tau}^*, \bm{w}}_{D^k} &= \LRp{\nabla u, \bm{w}}, &\forall \bm{w} \in \LRs{P^{N-1}(D^k)}^d, \label{eq:tau_test}\\
\LRa{\bm{\tau}^* \cdot\bm{n}, {w}}_{\reftwo{f_i}} &= \LRa{0, w}_{\reftwo{f_i}},  &\forall w \in P^{N}(\partial D^k \setminus f^*) \quad i = 1, \ldots, \reftwo{d} \nonumber \refone{,}
\end{align}
where $f^*$ denotes a single face for which $\beta = 1$ \reftwo{and $f_i$, $i = 1,...,d$ are all other faces of $D^k$}. Note that, if $\beta = \refone{\sgn{(\bm{v}_0\cdot \bm{n})}}$ for a locally \reftwo{non-zero} constant (or more generally, a divergence-free $\bm{v}_0$), then for affine simplices, at least one such face where $\beta = 1$ must exist \cite{cockburn2007analysis}. Next, we note that $\nor{\bm{\theta}}_{L^2(D^k)}$ can be expressed as a supremum \cite{john2016stable}
\[
\nor{\bm{\theta}}_{L^2(D^k)} = \sup_{\reftwo{\bm{w} \in \LRs{P^N{D^k} \setminus \{0\}}^d}} \frac{\LRp{\bm{\theta}, \bm{w}}_{D^k}}{\nor{\bm{w}}_{L^2(D^k)}}.
\]
Then, testing with $\bm{\tau}^*$ given by \eqref{eq:tau_test}, it follows that
\[
\nor{\bm{\theta}}_{L^2(D^k)} = \sup_{\reftwo{\bm{w} \in \LRs{P^N{D^k} \setminus \{0\}}^d}} \frac{\LRp{\bm{\theta}, \bm{w}}_{D^k}}{\nor{\bm{w}}_{L^2(D^k)}} \gtrsim \frac{\LRp{\bm{\theta}, \bm{\tau}^*}}{\nor{\bm{\tau}^*}_{L^2(D^k)}} \gtrsim \nor{\nabla u}_{L^2(D^k)},
\]
where we have used in the final step that $\LRp{\bm{\theta},\bm{\tau}^*}_{D^k} = \nor{\nabla u}_{L^2(D^k)}^2$ from \eqref{eq:tau_test} and that $\nor{\bm{\tau}^*}_{L^2(D^k)} \lesssim \nor{\nabla u}_{L^2(D^k)}$ by Lemma 3.2 in \cite{cockburn2007analysis}.
\end{proof}

\begin{remark}
While Lemma~\ref{lemma:ldg} only holds for simplices, the generalization to Cartesian quadrilateral or hexahedral elements can be performed using a tensor product version of this projection with similar properties; see (3.4-3.5) of \cite{cockburn2001superconvergence}. 
Additionally, the proof is straightforward to extend to curved elements if $\beta = 1$ on at least one face. 
\end{remark}

We note that Lemma~\ref{lemma:ldg} implies that the LDG gradient is locally null-space consistent, in that the nullspace of the LDG gradient operator is spanned only by the constant. This distinguishes LDG from other stabilized discretizations \cite{arnold2002unified}, which focus on null-space consistency of the discretized Laplacian as opposed to the gradient. We also note that, while Cockburn and Dong showed that the more general LDG fluxes where $\refone{\sgn{(\beta)} = \sgn{(\bm{v}_0 \cdot \bm{n})}}$ also result in a well-posed discretization of the Laplacian \cite{cockburn2007analysis}, they do not imply local null-space consistency of the gradient in general. 

\subsection{Upper bounds on the ECAV artificial viscosity coefficient}

Recall that the ECAV artificial viscosity coefficient is defined in \eqref{eq:ECAV_coeff} as 
\[
\epsilon_{k}(\bm{u}_h) = \frac{-\refone{\min{\left(0, \delta_k(\bm{u}_h)\right)}}}{\sum_{i,j=1}^d \LRp{\bm{K}_{ij}\bm{\Theta}_j, \bm{\Theta}_i}_{D^k}}
\refone{,}
\]
where the volume entropy residual $\delta_k(\bm{u}_h)$ is 
\[
\delta_k(\bm{u}_h) = \sum_{i=1}^d \int_{D^k} -\bm{f}_i(\bm{u}_h) \pd{\Pi_N \bm{v}(\bm{u}_h)}{x_i}  + \int_{\partial D^k} \psi_i(\tilde{\bm{u}})n_i.
\]
Here, \reftwo{the first term is the inner product between $f_i(\bm{u}_h)$ and the derivative of the projected entropy variables,} $\psi_i(\bm{u})$ is the entropy potential\reftwo{,} and $\tilde{\bm{u}} = \bm{u}(\Pi_N \bm{v}(\bm{u}_h))$ is the entropy projection (e.g., the evaluation of the conservative variables using the $L^2$ projection of the entropy variables). 

Using properties of the LDG gradient, we can bound the artificial viscosity from above. We first make the following assumption:
\begin{assumption}
\label{assumption:admissibility}
\reftwo{The solution $\bm{u}_h$ is defined such that the dissipation matrices $\bm{K}_{ij}$ are block matrices of a pointwise positive-definite matrix $\bm{K}$}
\[
\sum_{i,j=1}^d \bm{\Theta}_i^T \bm{K}_{ij}\bm{\Theta}_j \geq \lambda_{\min} \sum_{i=1}^d \bm{\Theta}_i^T\bm{\Theta}_i. 
\]
\end{assumption}
For the Laplacian artificial viscosity in this work, this corresponds to the condition that $\pd{\bm{u}}{\bm{v}}$ is \reftwo{uniformly positive-definite over each element $D^k$}. This holds if the entropy $S(\bm{u})$ is \reftwo{uniformly positive convex as defined in \cite{lukavcova2025note}}, and is equivalent to the density and internal energy being bounded uniformly from below by some positive value. 


\begin{lemma} \label{lemma: epsilon bound}
Let Assumption~\ref{assumption:admissibility} hold, and assume  exact integration of all integrals in the DG formulation \eqref{eq: viscous DG formulation} and that $\tilde{\bm{u}}, \bm{f}_m(\tilde{\bm{u}})$ are differentiable. \reftwo{Let $\bar{\bm{v}}$ be the integral mean of $\bm{v}(\bm{u}_h)$ over $D^k$ and $\bm{\delta v} = \bm{v}(\bm{u}_h) - \bar{\bm{v}}$. If $\Pi_N \bm{\delta v} \neq \bm{0}$ on $D^k$, then we have the following bound:}
\[
\epsilon_k(\bm{u}_h) \lesssim C(\bm{u}_h) h \refone{,}
\]
where $C(\bm{u}_h)$ is given by
\[
C(\bm{u}_h) = \nor{\lambda_{\min}^{-1}}_{L^\infty} \sqrt{\sum_{i=1}^d \nor{\pd{\bm{f}_i}{\bm{v}}}^2}
\sqrt{\frac{\nor{\bm{\delta v}}_{L^2(D^k)}^2}{\nor{\Pi_N \bm{\delta v}}_{L^2(D^k)}^2}-1}.
\]
Here, $\bar{\bm{v}}$ is the integral mean of $\bm{v}(\bm{u}_h)$ over $D^k$ and $\bm{\delta v} = \bm{v}(\bm{u}_h) - \bar{\bm{v}}$.
\end{lemma}
\begin{proof}
By Assumption~\ref{assumption:admissibility}, we have that
\[
\epsilon_k(\bm{u}_h) = \frac{-\min(0, \delta_k(\bm{u}_h))}{
\sum_{i,j=1}^d \LRp{\bm{K}_{ij}\bm{\Theta}_j, \bm{\Theta}_i}_{D^k}} 
\leq \nor{\lambda_{\min}^{-1}}_{L^\infty}\frac{\LRb{\delta_k(\bm{u}_h)}}{\nor{\bm{\theta}}^2_{L^2(D^k)}}.
\]
Lemma~\ref{lemma:ldg} then yields
\begin{equation}
\frac{\LRb{\delta_k(\bm{u}_h)}}{\nor{\bm{\theta}}_{L^2(D^k)}^2} \leq \frac{\LRb{\delta_k(\bm{u}_h)}}{\nor{\nabla \Pi_N \bm{v}(\bm{u}_h)}^2_{L^2(D^k)}}.
\label{eq:lemma1}
\end{equation}
We can bound the numerator $\delta_k(\bm{u}_h)$ by noting that, since $\tilde{\bm{u}}$ is defined in terms of the $L^2$ projection of the entropy variables \cite{chan2025artificial, colombo2022entropy}, the solution is differentiable and from \eqref{eq:cell_entropy_equality} the boundary integral of the entropy variables is equivalent to the following volume integral
\[
\int_{\partial D^k} \psi_i(\tilde{\bm{u}})n_i = \int_{D^k} \refone{\bm{f}_i(\tilde{\bm{u}})^T\pd{\Pi_N\bm{v}(\bm{u}_h)}{x_i}}. 
\]
Thus, factoring out derivatives of the projected entropy variables and using Cauchy-Schwarz yields
\[
\LRb{\delta_k(\bm{u}_h)} \leq \sqrt{\sum_{i=1}^d\nor{\bm{f}_i(\bm{u}_h) -\bm{f}_i(\tilde{\bm{u}})}_{L^2(D^k)}^2} \nor{\nabla \Pi_N \bm{v}(\bm{u}_h)}_{L^2(D^k)}  \refone{.}
\]
Substituting this into \eqref{eq:lemma1}, applying the mean value theorem, and factoring out the Jacobians $\pd{\bm{f}_i}{\bm{v}}$ then yields
\[
\epsilon_k(\bm{u}_h) \leq \sqrt{\sum_{i=1}^d \nor{\pd{\bm{f}_i}{\bm{v}}}^2}\frac{\nor{\bm{v}(\bm{u}_h) -\Pi_N\bm{v}(\bm{u}_h)}_{L^2(D^k)}}{\nor{\nabla \Pi_N \bm{v}(\bm{u}_h)}_{L^2(D^k)}}\refone{.}
\]
We  bound the denominator from below using a discrete Poincare-Friedrichs inequality \cite{brenner2003poincare, carstensen2018constants} 
\[
h\nor{\nabla \Pi_N \bm{v}(\bm{u}_h)}_{L^2(D^k)} \gtrsim  \nor{\Pi_N\bm{v}(\bm{u}_h) - \bar{\bm{v}}}_{L^2(D^k)}.
\]
Next, we note that the $L^2$ projection preserves constant moments, such that
\[
\nor{\bm{v}(\bm{u}_h) -\Pi_N\bm{v}(\bm{u}_h)}_{L^2(D^k)} = \nor{\LRp{\bm{v}(\bm{u}_h) - \bar{\bm{v}}} -\Pi_N\LRp{\bm{v}(\bm{u}_h)-\bar{\bm{v}}}}_{L^2(D^k)} = \nor{\LRp{\bm{\delta v} -\Pi_N\bm{\delta v}}}_{L^2(D^k)}.
\]
The proof is completed by noting that, by $L^2$ orthogonality of the projection error, $\nor{u}_{L^2}^2 = \nor{u - \Pi_Nu}_{L^2}^2 + \nor{\Pi_Nu}_{L^2}^2$ for any $u \in L^2$, and that
\[
\frac{\nor{u - \Pi_Nu}_{L^2}^2}{\nor{\Pi_Nu}_{L^2}^2} = \frac{\nor{u}_{L^2}^2 - \nor{\Pi_Nu}_{L^2}^2}{\nor{\Pi_Nu}_{L^2}^2} = \frac{\nor{u}_{L^2}^2}{\nor{\Pi_Nu}_{L^2}^2} - 1.
\]
\end{proof}
We note that, since the $L^2$ projection is contractive in the $L^2$ norm, the ratio of norms $\frac{\nor{\bm{\delta v}}_{L^2(D^k)}^2}{\nor{\Pi_N \bm{\delta v}}_{L^2(D^k)}^2} \geq 1$. However, numerical experiments suggest that this quantity is typically very close to $1$, even in the presence of shocks or near-zero density or pressure (which result in the entropy variable mapping being nearly singular and ill-conditioned). 

\begin{remark}[On quadrature accuracy]
We also note that, while Lemma~\ref{lemma: epsilon bound} assumes exact integration, for sufficiently accurate quadrature\footnote{For example, if the volume quadrature is exact for degree $2N+1$ polynomials and the surface quadrature exact for degree $2N+2$ polynomials \cite{chan2025artificial}.}, the quadrature error is of the same order of magnitude as the volume entropy residual $\delta_k(\bm{u}_h)$ and does not change the final estimate significantly. Moreover, we emphasize that the ECAV coefficient \eqref{eq:ECAV_coeff} is designed so that the formulation \eqref{eq: viscous DG formulation} is guaranteed to satisfy an entropy inequality independent of the choice of quadrature, and that quadrature accuracy affects only the magnitude of $\epsilon_k(\bm{u}_h)$. 
\end{remark}
    

\begin{remark}
Accounting for quadrature accuracy is also necessary when analyzing systems of conservation laws under the square entropy $S(\bm{u}) = \frac{1}{2}\bm{u}^T\bm{u}$, as $\bm{v}(\bm{u}) = \bm{u}$, as the entropy variables are identical to the conservative variables and the factor $\nor{\bm{v}(\bm{u}_h) - \Pi_N\bm{v}(\bm{u}_h)}_{D^k}$ in Lemma~\ref{lemma: epsilon bound} would vanish. This reflects the fact that DG discretizations of systems which are symmetrizable by a square entropy are automatically entropy stable under exact integration. Thus, while the bound in Lemma~\ref{lemma: epsilon bound} still holds for these simpler systems, it is overly pessimistic and not practically useful for such settings. 
\end{remark}

\reftwo{\begin{remark}[]
    When $\Pi_N \bm{\delta v} = \bm{0}$, the first inner product in the volume entropy residual \eqref{eqn: semi-discrete volume entropy residual} vanishes because $\Pi_N \bm{v}(\bm{u}_h) = \bar{\bm{v}}$, which implies $\nabla \Pi_N \bm{v}(\bm{u}_h) = \bm{0}$. The second inner product vanishes because $\tilde{\bm{u}} = \bm{u}(\Pi_N \bm{v}(\bm{u}_h))$ is constant, meaning that $\LRa{\psi_m(\bm{\tilde{u}})n_m,1} = \psi_m(\bm{\tilde{u}})\LRa{n_m,1}$. Adding together the surface integrals of the outward unit normals over $\partial D^k$ yields zero. Therefore, the volume entropy residual vanishes ($\delta_k(\bm{u}_h) = 0$), and consequently the artificial viscosity satisfies $\epsilon_k(\bm{u}_h) = 0$.
\end{remark}}

Finally, we note that Lemma~\ref{lemma: epsilon bound} provides a pessimistic upper bound on the magnitude of $\epsilon_k(\bm{u}_h)$. However, for sufficiently regular solutions, we expect $\epsilon_k(\bm{u}_h)$ to be much smaller in magnitude. The beginning of the proof of Lemma~\ref{lemma: epsilon bound} implies that
\[
\epsilon_k(\bm{u}_h) = \frac{-\min(0, \delta_k(\bm{u}_h))}{\int_{D^k} \bm{\theta}^T\bm{K}\bm{\theta}} \lesssim \frac{\LRb{\delta_k(\bm{u}_h)}}{\nor{\nabla \Pi_N \bm{v}(\bm{u}_h)}^2_{D^k}}. 
\]
We now consider the behavior of $\epsilon_k(\bm{u}_h)$ as $\bm{u}_h \rightarrow \bm{u}$ under mesh refinement for some fixed sufficiently regular function $\bm{u}$. \reftwo{For regions where $\bm{u}$ is not constant, the denominator $\nor{\nabla \Pi_N \bm{v}(\bm{u}_h)}^2_{D^k}$ will converge to $Ch^d$ under mesh refinement, where $C$ is a constant depending on $\nabla \Pi_N \bm{v}(\bm{u}_h)$ and the reference element. Since the numerator $|\delta_k(\bm{u}_h)|$ is of magnitude $O(h^{2N+2+d})$, the ratio will converge to zero.} When $\bm{u}_h$ is exactly constant over $D^k$, $\epsilon_k(\bm{u}_h) = 0$ when calculated using the regularized ratio \eqref{eq:regularized_ratio} since both the denominator and numerator vanish. Thus, what remains is to consider the case when $\nabla\bm{u}$ vanishes only at specific points (or over a surface) within an element. For such cases, applying a discrete Poincare-Friedrichs inequality \cite{brenner2003poincare, carstensen2018constants} yields
\[
\epsilon_k(\bm{u}_h) \lesssim \frac{\LRb{\delta_k(\bm{u}_h)}}{\nor{\nabla \Pi_N \bm{v}(\bm{u}_h)}^2_{D^k}} \lesssim h^2 \frac{\LRb{\delta_k(\bm{u}_h)}}{\nor{\Pi_N (\bm{v}(\bm{u}_h) - \bar{\bm{v}})}^2_{D^k}}.
\]
It was shown in \cite{chan2025artificial} that the numerator is  $O(h^{2N+2+d})$ under exact integration. Assuming that $\nabla \bm{u}$ vanishes inside $D^k$ (but is not identically zero), $\bm{v}(\bm{u}_h) - \bar{\bm{v}} = O(h)$ and the denominator is $O(h^{2+d})$. Together, this suggests that $\max_k \epsilon_k(\bm{u}_h) \lesssim O(h^{2N+2})$ for sufficiently regular solutions, which is consistent with numerical results presented in \cite{chan2025artificial} for sufficiently accurate quadrature. 

\section{Numerical experiments}\label{sec:section_5}

In this section, we present numerical experiments which confirm theoretical properties of ECAV under an LDG discretization of \eqref{eq: monolithic viscous regularization}, compared to a BR-1 type regularization used in \cite{chan2025artificial}. All experiments are implemented in Julia using the \verb+StartUpDG.jl+ and \verb+Trixi.jl+ \cite{ranocha2022adaptive} libraries. For the following 1D experiments, nodal collocation formulations use a $(N+1)$-point Gauss-Lobatto quadrature rule, while modal formulations use a $(N+2)$-point Gauss-Legendre quadrature rule. Unless otherwise stated, we use the HLLC interface flux \cite{batten1997choice}. Finally, the parameter $\beta$ for the LDG discretization in \Cref{eq: LDG 1} and \Cref{eq: LDG 3} is chosen as $\beta = \text{sgn}(n_x)$ and $\beta = \text{sgn}(2n_x + n_y)$ for 1D and 2D experiments, respectively. 

For time integration, we use the \verb+OrdinaryDiffEq.jl+ library \cite{rackauckas2017differentialequations}. Unless stated otherwise, we use the adaptive 4-stage 3rd order strong stability preserving Runge-Kutta method (SSPRK43) \cite{FCS2022, kraaijevanger1991contractivity, ranocha2022optimized} with absolute and relative tolerances set to $10^{-6}$ and $10^{-4}$, respectively. 

Finally, we will consider two different formulations which result from different choices of quadrature. Unless otherwise specified, we use a ``modal'' formulation with a $(N+2)$-point Gauss-Legendre quadrature in 1D, and quadrature rules which are exact for degree $2N$ polynomials on triangular elements. 

\subsection{2D Inviscid Burgers Equation}\label{Example: Burgers Equation}
We begin by comparing the magnitude of the entropy correction artificial viscosity coefficient using the BR-1 and LDG discretizations. Consider the 2D inviscid Burgers Equation
\begin{equation}\label{eq: 2D Burgers}
    \pd{u}{t} + \pd{}{x}\left(\frac{1}{2}u^2\right) + \pd{}{y}\left(\frac{1}{2}u^2 \right) = 0.
\end{equation}
An initial boundary value problem is formulated on a domain of $(x, y) \in [-1, 1]^2
$, with periodic boundary conditions, a Gaussian initial condition
\begin{equation*}
    u_0(x,y) = \exp{\left(-25(x^2 + y^2)\right)}, 
\end{equation*}
and a final time of $t = 0.77$. \Cref{eq: 2D Burgers} admits the following entropy functional with the associated entropy potentials
\begin{equation*}
    S(u) = \frac{u^2}{2}, \quad \psi_1(u) = \psi_2(u) = \frac{u^3}{6}.
\end{equation*}
We use polynomial order $N = 3$, $K = 2 \times 30^2$ triangular elements, and the entropy-conserving interface flux \cite{gassner2013skew}
\begin{equation*}
    f^*_{\bm{n}} (u, u^+) = \frac{1}{6}\left((u^+)^2 + u^{-}u^{+} + (u^-)^2\right) (n_1 + n_2).
\end{equation*}
In \Cref{fig:Burgers Example Entropy} we observe that the rate of change of entropy is non-positive, implying that both entropy correction methods satisfy the entropy inequality. However, \reftwo{in \Cref{fig:Burgers Example Epsilon and Time Step} we see that} the maximum artificial viscosity discretized with BR-1 is several orders of magnitude larger than the maximum artificial viscosity discretized with LDG. As a result, an adaptive SSPRK43 time integrator results in 474 adaptive time steps for LDG, compared to \reftwo{21084} adaptive time steps for BR-1. 
\begin{figure}[tbp]
\centering
\begin{subfigure}[t]{.49\textwidth}
\centering
    \includegraphics[width =\textwidth]{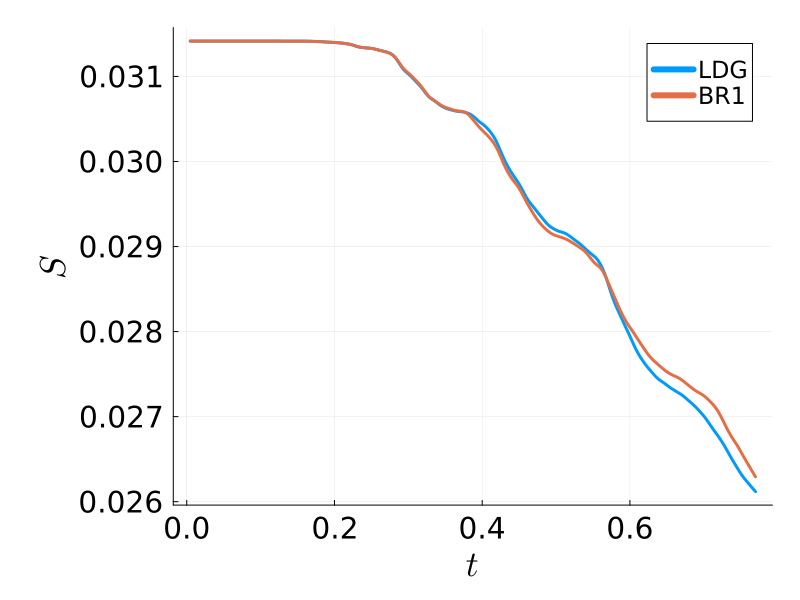}
    \caption{\centering Entropy, $S$}
    \label{fig:Burgers 2D Entropy}
\end{subfigure}
\hfill
\begin{subfigure}[t]{0.49\textwidth}
\centering
     \includegraphics[width = \textwidth]{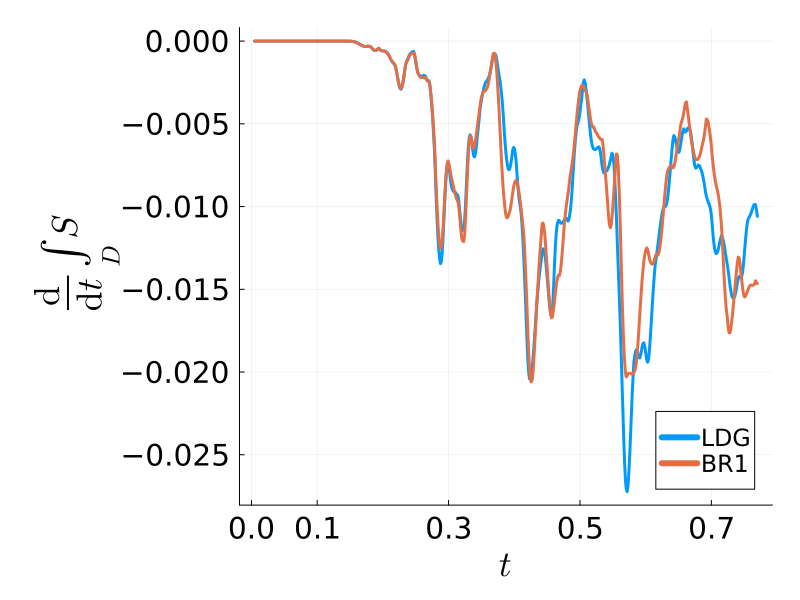}
     \caption{Change in entropy, $\frac{\mathrm{d}}{\mathrm{d}t}\int_{D}S$}
    \label{fig:Burgers 2D Change in Entropy}
\end{subfigure}
\caption{Example \ref{Example: Burgers Equation}, time evolution of the \reftwo{entropy} and change in entropy over time for BR-1 and LDG discretizations of the entropy correction artificial viscosity with \reftwo{$N = 3$} and \reftwo{$K = 2 \times 30^2$} elements.}
\label{fig:Burgers Example Entropy}
\end{figure}

\begin{figure}[tbp]
\centering
\begin{subfigure}[t]{.49\textwidth}
\centering
    \includegraphics[width =\textwidth]{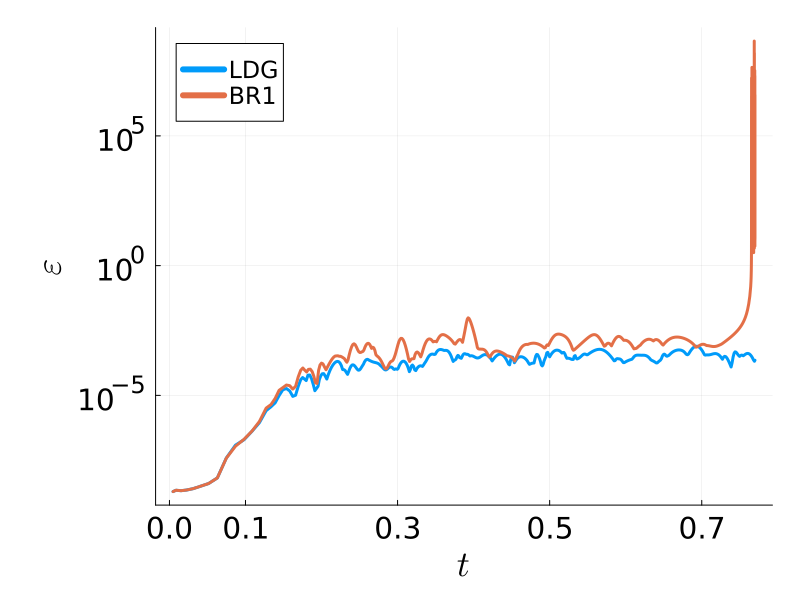}
    \caption{\centering $\displaystyle\max_k \epsilon_k$}
    \label{fig:Burgers 2D epsilon}
\end{subfigure}
\hfill
\begin{subfigure}[t]{0.49\textwidth}
\centering
     \includegraphics[width = \textwidth]{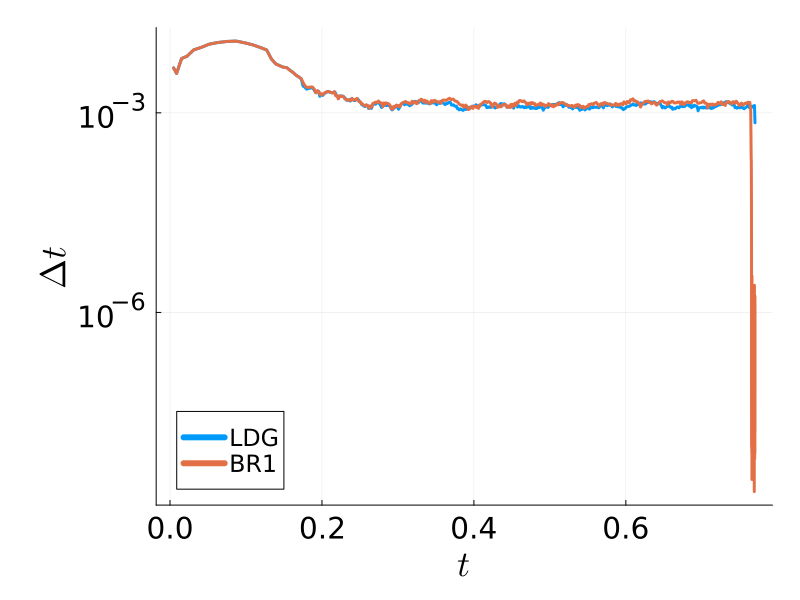}
     \caption{Time Step, $\Delta t$}
    \label{fig:Burgers 2D Time Step}
\end{subfigure}
\caption{Example \ref{Example: Burgers Equation}, time evolution of the $\max_k{\epsilon_k}$ (semi-log plot) and time step size (semi-log plot) for BR-1 and LDG discretizations of the entropy correction artificial viscosity with \reftwo{$N = 3$} and \reftwo{$K = 2 \times 30^2$} elements.}
\label{fig:Burgers Example Epsilon and Time Step}
\end{figure}

We also analyzed the behavior of BR-1 and LDG under a nodal DG-like formulation constructed using multi-dimensional summation-by-parts (SBP) operators \cite{hicken2016multidimensional}. Under this nodal SBP formulation, LDG completes the simulation in 640 adaptive time steps, while the BR-1 discretization does not converge even after $100,000$ time steps. We note that, due to $\epsilon_k(\bm{u}_h)$ being large, BR-1 induces extremely small time steps on the order of $10^{-10}$ starting at roughly $t = 0.39897$. 
\begin{remark}\label{remark: regularized ratio}
     Under the BR-1 viscous discretization and a nodal DG formulation, computing the regularized ratio \eqref{eq:regularized_ratio} with $\delta = 10^{-14}$ resulted in violations of the semi-discrete entropy inequality due to finite precision effects. To address this issue, for both the nodal and modal DG formulations in this section, we use \verb+eps(b)+ for $\delta$, which, \reftwo{in Julia,} returns the distance between \verb+b+ and the next representable floating-point value larger than \verb+b+. \reftwo{Under the LDG viscous discretization, we note that choosing $\delta = 10^{-6}, 10^{-14}$ or {\upshape\texttt{eps(b)}} gives solutions whose differences are negligible.}
\end{remark}

We do not observe a significant difference between BR-1 and LDG for this problem when using an entropy stable interface flux. However, while the use of an entropy conservative interface flux in these experiments is mainly intended to confirm our theoretical results, the use of non-dissipative interface fluxes in realistic simulations has been explored for applications in turbulence modeling \cite{flad2017use}.

\subsection{Compressible Euler Equations}\label{sec:Compressible Euler Equations}
The remaining examples solve the 1D or 2D compressible Euler equations described below. 
\reftwo{For detailed expositions on multi-dimensional formulations, entropy pairs, and 
entropy variable transformations, we refer the reader to \cite{ harten1983symmetric, hughes1986new, tadmor2003entropy} 
(see also \cite{chan2022entropy} for high-order DG discretizations). The numerical experiments for the compressible Euler equations were implemented using the open-source Julia framework Trixi.jl \cite{ranocha2022adaptive}, utilizing its built-in formulations for physical fluxes, entropy variables, and entropy conservative/dissipative interface fluxes.} 

Let $\bm{u}$ denote the vector of conserved variables, which in 2D are
\begin{equation*}
    \bm{u} = \left\{\rho,\rho u_1, \rho u_2, E \right\} \in \mathbb{R}^4.
\end{equation*}
Here, $\rho$ is the density, $u_1$ and $u_2$ are the velocity in the $x$ and $y$ direction, and $E$ is the specific total energy. The pressure $p$ is related to the conservative variables through the ideal gas constitutive relations
\begin{equation*}
    p = (\gamma-1)\rho e, \quad E = \reftwo{\rho} e + \frac{1}{2}(\reftwo{\rho} u_1^2 + \reftwo{\rho} u_2^2),
\end{equation*}
where $\gamma$ is the ratio of specific heats and $e$ is the internal energy density. All problems in this work utilize $\gamma = 1.4$ for calorically perfect air. In 2D the compressible Euler equations are 
\begin{equation*}
    \pd{\bm{u}}{t} + \pd{\bm{f}_1}{x} + \pd{\bm{f}_2}{y} = \bm{0},
\end{equation*}
where $\bm{f}_1$ and $\bm{f}_2$ denote the convective fluxes in the $x$ and $y$ direction,
\begin{equation*}
    \bm{f}_1 = 
    \begin{bmatrix}
        \rho u_1\\
        \rho u_1^2 + p\\
        \rho u_1 u_2 \\
        u_1(E+p)
    \end{bmatrix}, \quad 
    \bm{f}_2 = 
        \begin{bmatrix}
        \rho u_2\\
        \rho u_1 u_2\\
        \rho u_2^2 + p \\
        u_2(E+p)
    \end{bmatrix}.
\end{equation*}
If the velocity is set to zero in the $y$-direction, the 1D compressible Euler equations are recovered by disregarding the $y$-momentum equation.

The convex entropy function $\reftwo{S(\bm{u})}$ and corresponding entropy potentials $\psi_1(\bm{u})$ and $\psi_2(\bm{u})$
\begin{equation*}
    S(\bm{u}) = -\rho s, \qquad \psi_1(\bm{u}) = \rho u_1, \qquad \psi_2(\bm{u}) = \rho u_2,
\end{equation*}
where $s = \log{\left(\frac{p}{\rho^{\gamma}} \right)}$
denotes the physical entropy. The derivative of the entropy with respect to the conservative variables yield the entropy variables $\bm{v}(\bm{u}) = \pd{S}{\bm{u}} = \left\{v_1, v_2, v_3, v_4 \right\}$, where
\begin{equation*}
    v_1 = \frac{\rho e \left(\gamma + 1 - s\right) \reftwo{-E}}{\rho e}, \quad v_2 = \frac{\rho u_1}{\rho e}, \quad v_{3} = \frac{\rho u_2}{\rho e}, \quad v_4 = -\frac{\rho}{\rho e}.
\end{equation*}
The inverse mapping is given by
\begin{equation*}
    \rho = -(\rho e)v_4, \quad \rho u_1 = (\rho e)v_2, \quad \rho u_2 = (\rho e) v_3, \quad E = (\rho e) \left(1 - \frac{v_2^2 + \reftwo{v_3^{2}}}{2v_4} \right),
\end{equation*}
and $\rho e$ and $s$ in terms of the entropy variables are
\begin{equation*}
    \rho e = \left(\frac{(\gamma -1)}{(-v_4)} \right)^{1/(\gamma - 1)} e^{\frac{-s}{\gamma - 1}}, \quad s = \gamma - v_1 + \frac{v_2^2 + \reftwo{v_3^2}}{2v_4}.
\end{equation*}
Finally, explicit expressions for the Jacobian matrix $\pd{\bm{u}}{\bm{v}}$ are given in terms of the sound speed $a$ and specific total enthalpy $H$ \cite{barth1999numerical}. 
\begin{equation*}
\begin{gathered}
    \pd{\bm{u}}{\bm{v}} = 
    \begin{bmatrix}
        \rho & \rho u_1       & \rho u_2     & E\\
             & \rho u_1^2 + p & \rho u_1 u_2 & u_1 (E + p)\\
             & & \rho u_2^2 + p & u_2(E + p) \\
             & & & \rho H^2 - a^2\frac{p}{\gamma -1}
    \end{bmatrix},\\
    a = \sqrt{\gamma \frac{p}{\rho}}, \quad H = \frac{a^2}{\gamma -1} + \frac{1}{2}(u_1^2 + u_2^2),
\end{gathered}
\end{equation*}
where the lower triangular entries of $\pd{\bm{u}}{\bm{v}}$ are determined by symmetry.

\subsection{Shu Isentropic Euler Vortex Problem}\label{Ex: Vortex Problem}
This problem appeared in \cite{Shu1998Vortex} to compare high-order weighted essentially non-oscillatory \\
(WENO) finite-volume (FV) methods with a traditional second-order FV method. Since then, many variations of the problem have been considered in the literature; see \cite{Spiegal2015VortexProblem} and the references therein. This problem has an exact solution and we will use it to examine the accuracy of the proposed method. To reduce the effect of temporal discretization errors in our convergence results, we use the 5th-order accurate Runge-Kutta method of \cite{TSITOURAS2011770} (\verb+Tsit5+ in \verb+OrdinaryDiffEq.jl+) instead of SSPRK43. The absolute tolerance is set to $10^{-10}$ and the relative tolerance is set to $10^{-8}$.

The initial conditions for this problem are an isentropic vortex added to mean flow quantities. The mean flow quantities are $\rho = p = 1$ and $u_1 = u_2 = 1/\gamma$ and the isentropic vortex is added by perturbing $u_1$, $u_2$ and the temperature $T = p / \rho$, but not perturbing the entropy $S = p/\rho^{\gamma}$. These perturbations are
\begin{equation*}
\begin{gathered}
    \delta u_1(x,y) = -y{\Omega(x,y)}, \qquad \delta u_2(x,y) = x{\Omega(x,y)}, \\
    \delta T(x,y) = -\frac{(\gamma - 1)}{2}\Omega^{\reftwo{2}}(x,y), \qquad
    \Omega(x,y) = \alpha \exp{-\frac{1}{2}\left(x^2 + y^2\right)},
\end{gathered}
\end{equation*}
where $\alpha = 5 \exp{(1/2)} / (2 \pi \sqrt{\gamma})$ is the maximum strength of the perturbation. The analytical solution is a translation of the initial condition given by
\begin{equation}
\begin{aligned}\label{eq: isentropic solution}
    \rho(x,y,t) &= \tilde{\rho}\left(x - \frac{1}{\gamma}t, y - \frac{1}{\gamma} t\right), \quad \tilde{\rho}(x,y) = \left(1 + \delta T(x, y)\right)^{1/ (\gamma - 1)},\\
    u_1(x,y,t) &= \tilde{u}_1\left(x - \frac{1}{\gamma}t, y - \frac{1}{\gamma} t\right), \quad \tilde{u}_1(x,y) = \frac{1}{\gamma} + \delta u_1(x,y),\\
    u_2(x,y,t) &= \tilde{u}_2\left(x - \frac{1}{\gamma}t, y - \frac{1}{\gamma} t\right), \quad \tilde{u}_2(x,y) = \frac{1}{\gamma} + \delta u_{\reftwo{2}}(x,y),\\
    p(x,y,t) &= \tilde{p}\left(x - \frac{1}{\gamma}t,y-\frac{1}{\gamma}t\right), \quad \tilde{p}(x,y) = \left(\tilde{\rho}(x,y)\right)^\gamma.
\end{aligned}
\end{equation}

In \Cref{table:Shu Isentropic Vortex}, we present the errors and rates of convergence for $N = 1,2,3$ and $4$ using a modal DG formulation. The convergence rates are higher than the best $L^2$ approximation convergence rates, but we anticipate that under further mesh refinement, these rates will decrease and approach $N+1$. 
We note that similar errors and rates of convergence are obtained using both a BR-1 discretization of the artificial viscosity, as well as using a standard DG method with no artificial viscosity. 

\begin{table}[tbp]
\centering

\begin{tabular}{|c|c|c|c|c|}
\hline
$K$ & $L^2$ (N=1) & Order & $L^2$ (N=2) & Order \\
\hline
$16$ & $3.134 \times 10^{ -2 }$ & --- & $7.770 \times 10^{ -3 }$ & --- \\
$32$ & $1.201 \times 10^{ -2 }$ & $1.3834$ & $5.854 \times 10^{ -4 }$ & $3.7304$ \\
$64$ & $2.290 \times 10^{ -3 }$ & $2.3914$ & $4.845 \times 10^{ -5 }$ & $3.5949$ \\
$128$ & $3.522 \times 10^{ -4 }$ & $2.7006$ & $5.261 \times 10^{ -6 }$ & $3.2031$ \\
\hline
$K$ & $L^2$ (N=3) & Order & $L^2$ (N=4) & Order \\
\hline
$16$ & $1.274 \times 10^{ -3 }$ & --- & $3.567 \times 10^{ -4 }$ & --- \\
$32$ & $8.356 \times 10^{ -5 }$ & $3.9308$ & $1.314 \times 10^{ -5 }$ & $4.7620$ \\
$64$ & $4.201 \times 10^{ -6 }$ & $4.3141$ & $3.542 \times 10^{ -7 }$ & $5.2135$ \\
$128$ & $2.176 \times 10^{ -7 }$ & $4.2706$ & $9.870 \times 10^{ -9 }$ & $5.1656$ \\
\hline
\end{tabular}

\caption{Example \ref{Ex: Vortex Problem}: A h-refinement convergence table. $2 \times K^2$ is the number of uniform triangular elements. The relative $L_2$ error and approximate order of accuracy are also given for the corresponding values of $K$.}
\label{table:Shu Isentropic Vortex}
\end{table}

\subsection{Stationary Contact Wave and Contact Preservation} \label{ex: stationary contact wave}

In the following stationary contact wave example, we demonstrate that ECAV is contact-preserving. Note that for piecewise constant initial conditions, the entropy correction artificial viscosity vanishes. The volume integral in the entropy residual \eqref{eqn: semi-discrete volume entropy residual} vanishes because $\reftwo{\Pi_N \bm{v}(\bm{u_h})}$ is constant on each element. Similarly, since $\psi_m(\tilde{\bm{u}})$ is constant over each element, the boundary integral in the entropy residual reduces to the integral over $\partial D^k$ of the outward normal component of a constant vector. Thus, if a baseline DG formulation is contact-preserving (e.g., if a contact preserving interface flux $\bm{f}^*_n$ is used) then \eqref{eq: viscous DG formulation} is contact preserving. 

To verify this, we consider the piecewise constant initial conditions
\begin{equation}\label{eq: stationary contact initial conidions constant.}
(\rho, u, p) = \begin{cases}
    (1.5, 0,1) & |x| < 0.3\\
    (1,0,1) & |x| \geq 0.3,
\end{cases}
\end{equation}
and the piecewise smooth initial condition
\begin{equation}\label{eq: stationary contact initial conidions not constant.}
(\rho, u, p) = \begin{cases}
    (1 + 0.5(\sin{(2\pi x)}+|x|), 0,1) & |x| < 0.3\\
    (1 + \reftwo{0.5}\sin{(2\pi x)},0,1) & |x| \geq 0.3,
\end{cases}
\end{equation}
with a domain $[-1.0,1.0]$, periodic boundary conditions, and a final time of $t = 4.0$. Since the piecewise constant initial conditions given by \eqref{eq: stationary contact initial conidions constant.} are exactly representable by a DG basis, we expect to observe exact contact preservation in this case since the interface flux is the contact-preserving HLLC flux \cite{batten1997choice}.

In \Cref{fig:Stationary Contact Wave Example}, we validate the contact preserving property of the proposed method by tracking the time evolution of the $L^2$ error for $N = 2,4,6,8$, a fixed time step of $\Delta t = 5 \times 10^{-4}$, $K = 80$ elements, for both initial conditions \eqref{eq: stationary contact initial conidions constant.} and \eqref{eq: stationary contact initial conidions not constant.}. With piecewise constant initial conditions, the stationary contact is preserved up to machine precision. For the piecewise smooth solution, the stationary contact is preserved over time up to high-order accuracy. 
\begin{figure}[tbp]
\centering
\begin{subfigure}[t]{.49\textwidth}
\centering
    \includegraphics[width =\textwidth]{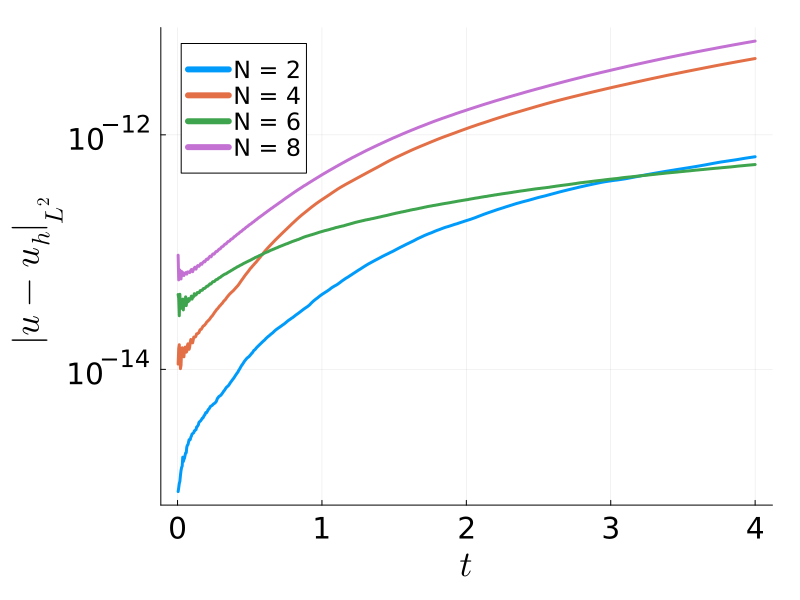}
    \caption{\centering Piecewise constant stationary solution \eqref{eq: stationary contact initial conidions constant.}}
    \label{fig: Stationary contact constant IC}
\end{subfigure}
\hfill
\begin{subfigure}[t]{0.49\textwidth}
\centering
     \includegraphics[width = \textwidth]{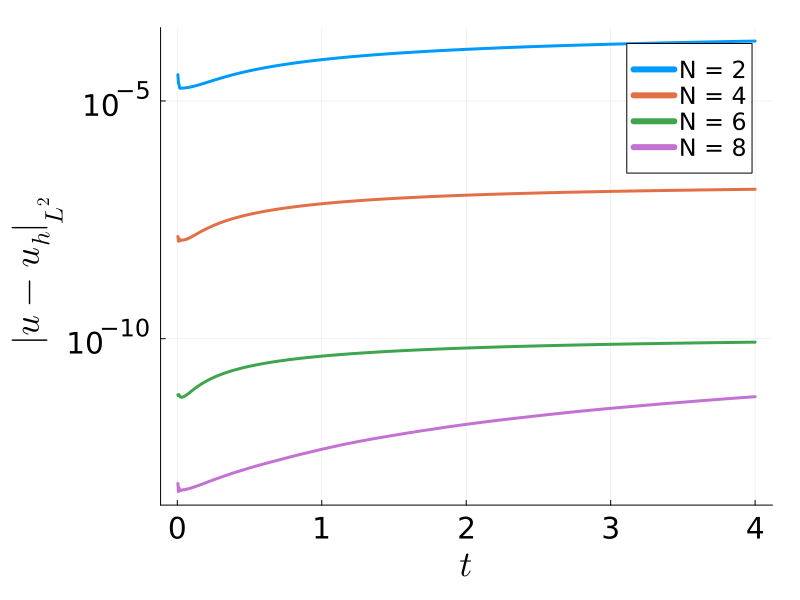}
     \caption{Piecewise smooth stationary solution \eqref{eq: stationary contact initial conidions not constant.}}
    \label{fig:Stationary contact noncostant IC}
\end{subfigure}
\caption{Example \ref{ex: stationary contact wave}, semi-log plots of the $L^2$ error evolution for two different stationary contact solutions. }
\label{fig:Stationary Contact Wave Example}
\end{figure}

\subsection{Shock-vortex interaction} \label{ex:Shock Vortex Interaction}

The next problem is the 2D shock-vortex interaction problem from \cite{Shu1998Vortex}. We will examine how the entropy correction LDG artificial viscosity method resolves the interaction between the shock and the vortex. The domain is $[0,2] \times [0,1]$ and a stationary Mach $M_s = 1.1$ shock is positioned at $x = 0.5$ and normal to the x-axis. On the left of the shock the initial condition is 
$
(\rho_L, u_L, v_L, p_L) = (1, \reftwo{1.1\sqrt{\gamma}},0,1)
$, and the initial conditions on the right side of the shock are computed using the Rankine-Hugoniot conditions 
$$
\frac{\rho_L}{\rho_R} = \reftwo{\frac{u_R}{u_L}} = \frac{2+(\gamma - 1) M_s^2}{(\gamma + 1)M_s^2}, \quad \reftwo{\frac{p_R}{p_L}} = 1 + \frac{2\gamma}{\gamma + 1} (M_s^2 - 1), \quad v_R = 0.
$$
An isentropic vortex centered at $(x_c,y_c) = (0.25,0.5)$ is added as a perturbation of $(u_1,u_2)$, and temperature $T = p/\rho$
\begin{gather*}
\delta u_1 = v_{\theta}\sin{(\theta)}, \qquad \delta u_2 = -v_{\theta}\cos{(\theta)}, \\ 
 v_{\theta} = \epsilon \tau \exp{(\alpha(1-\tau^2))}, \qquad \delta T = -\frac{(\gamma - 1)\epsilon^2\exp{(2\alpha(1-\tau^2))}}{4\alpha \gamma} \refone{,}
\end{gather*}
where $r = \sqrt{(x-x_c)^2 + (y-y_c)^2}$, $\tau = r/r_c$, $\theta = \tan^{-1}(\frac{y-y_c}{x-x_c})$, $\epsilon = 0.3$, $\alpha = 0.204$, and $r_c = 0.05$. Finally, entropy $S = \ln{(p/\rho^{\gamma})}$ is not perturbed and our initial conditions are 
$$
(\rho,u_1,u_2,p) = \begin{cases} \left(\rho_L\left(\frac{T_L + \delta T}{T_L}\right)^{\frac{1}{\gamma-1}}, \; u_L + \delta u_1, \; v_L + \delta u_2, \; p_L\left(\frac{T_L + \delta T}{T_L}\right)^{\reftwo{\frac{\gamma}{\gamma-1}}}\right) & x < 0.5\\
\left(\rho_R\left(\frac{T_L + \delta T}{T_L}\right)^{\frac{1}{\gamma-1}}, \; u_R + \delta u_1, \; v_R + \delta u_2, \; p_R\left(\frac{T_L + \delta T}{T_L}\right)^{\reftwo{\frac{\gamma}{\gamma-1}}}\right) & x > 0.5
\end{cases}
$$

\reftwo{We enforce free-stream boundary conditions on the left boundary, subsonic outflow boundary conditions on the right boundary, and wall boundary conditions on the top and bottom boundaries.} 
Using a $128 \times 64$ element mesh and $N = 2$, in \Cref{fig:Shock Vortex Interaction} we plot the density profile and the numerical Schlieren data at times $t = 0$, $t = 0.2$ and $t = 0.7$. The numerical Schlieren plot visualizes gradients of the density field by plotting the quantity:
\begin{equation}\label{eq:rho Schlieren}
    \rho^{schl} = \exp{\left(-10 \frac{g - g_{\min}}{g_{\max} - g_{\min}}\right)}, \quad g = \left\|\nabla \rho \right\|, \quad g_{\min} = \min_{x \in D}{g(x)}, \quad g_{\max} = \max_{x \in D}{g(x)}.
\end{equation}
In the numerical Schlieren plots, we observe that while oscillations are present, they remain localized around the vicinity of the shock. 

\begin{figure}[tbp]
\centering
\begin{subfigure}[t]{.49\textwidth}
\centering
    \includegraphics[width =\textwidth]{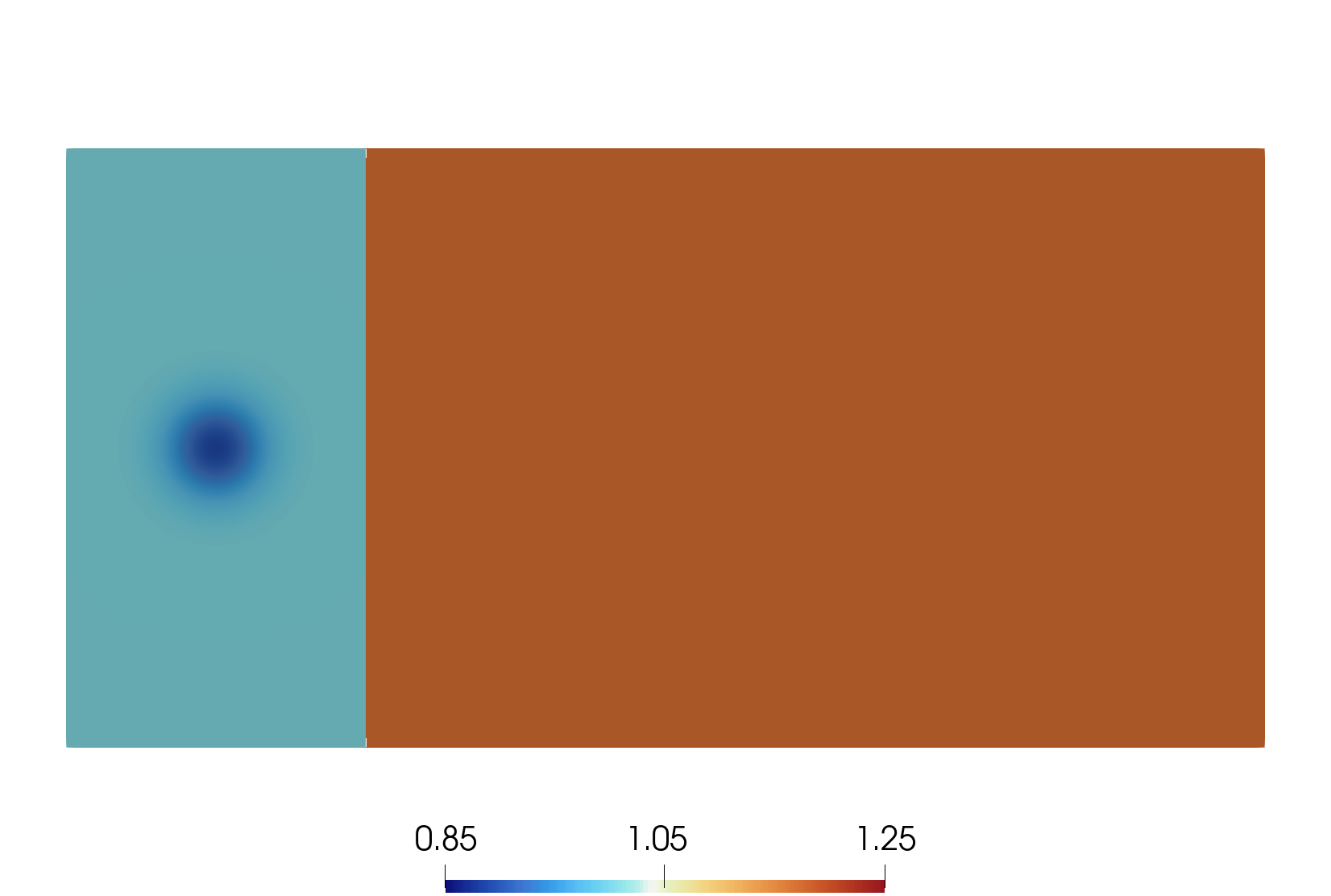}
    \caption{\centering Density at $t = 0$.}
    \label{fig: Schock Vortex interaction Density t =0}
\end{subfigure}
\hfill
\begin{subfigure}[t]{0.49\textwidth}
\centering
     \includegraphics[width = \textwidth]{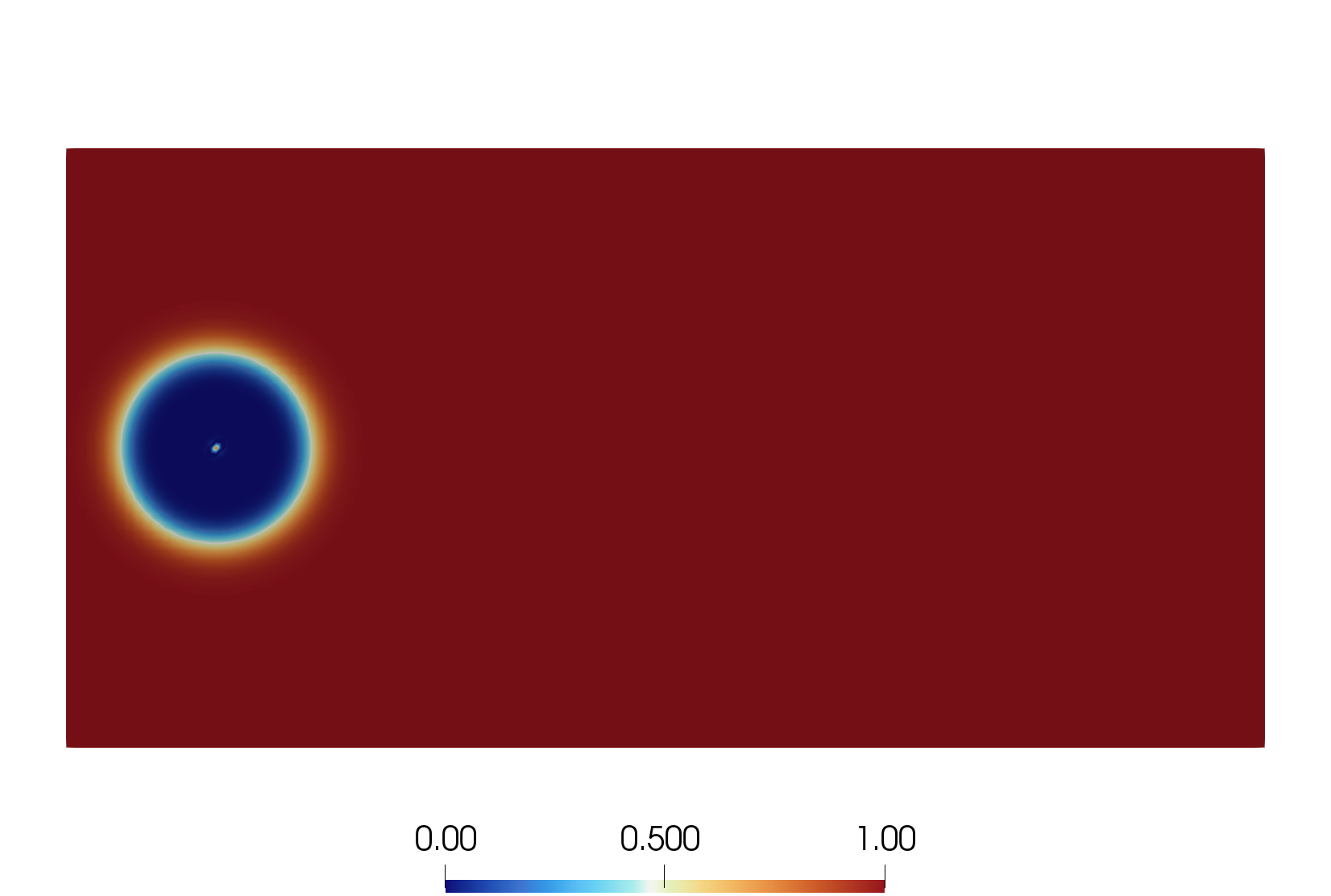}
     \caption{Schlieren plot at $t = 0$.}
    \label{fig: Schock Vortex interaction Schlieren t = 0}
\end{subfigure}

\begin{subfigure}[t]{.49\textwidth}
\centering
    \includegraphics[width =\textwidth]{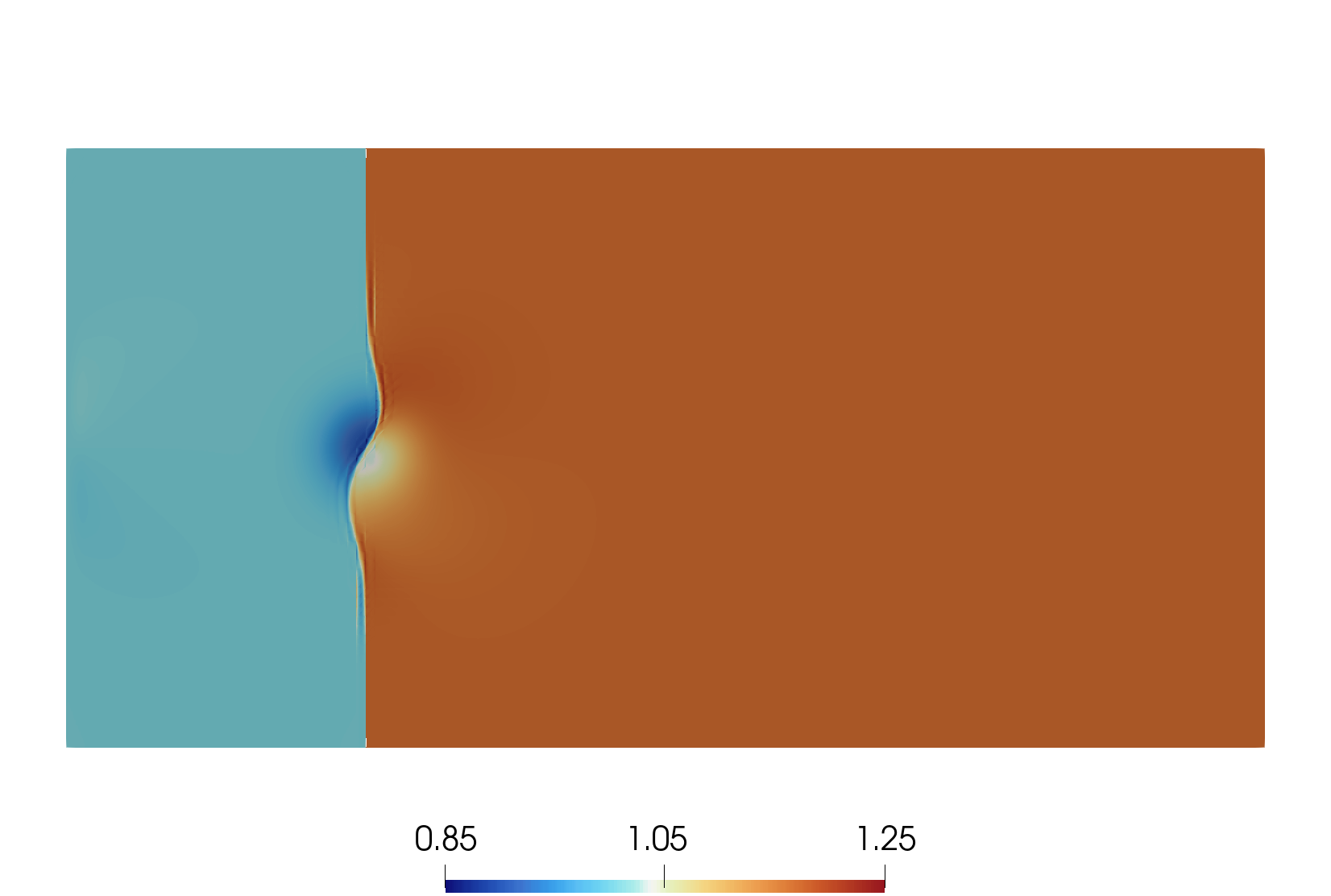}
    \caption{\centering Density at $t = 0.2$.}
    \label{fig: Schock Vortex interaction Density t =0.2}
\end{subfigure}
\hfill
\begin{subfigure}[t]{0.49\textwidth}
\centering
     \includegraphics[width = \textwidth]{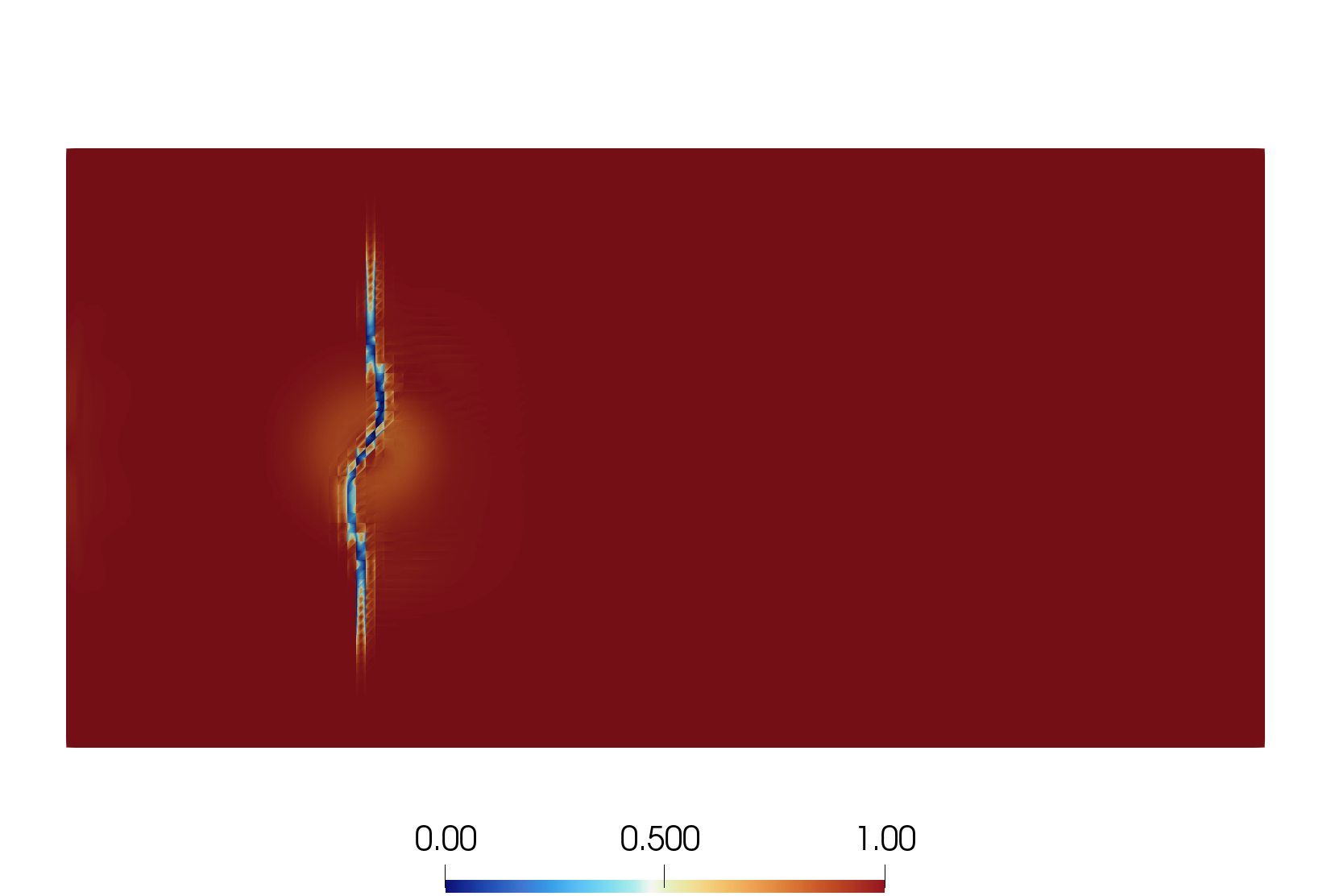}
     \caption{Schlieren plot at $t = 0.2$.}
    \label{fig: Schock Vortex interaction Schlieren t = 0.2}
\end{subfigure}

\begin{subfigure}[t]{.49\textwidth}
\centering
    \includegraphics[width =\textwidth]{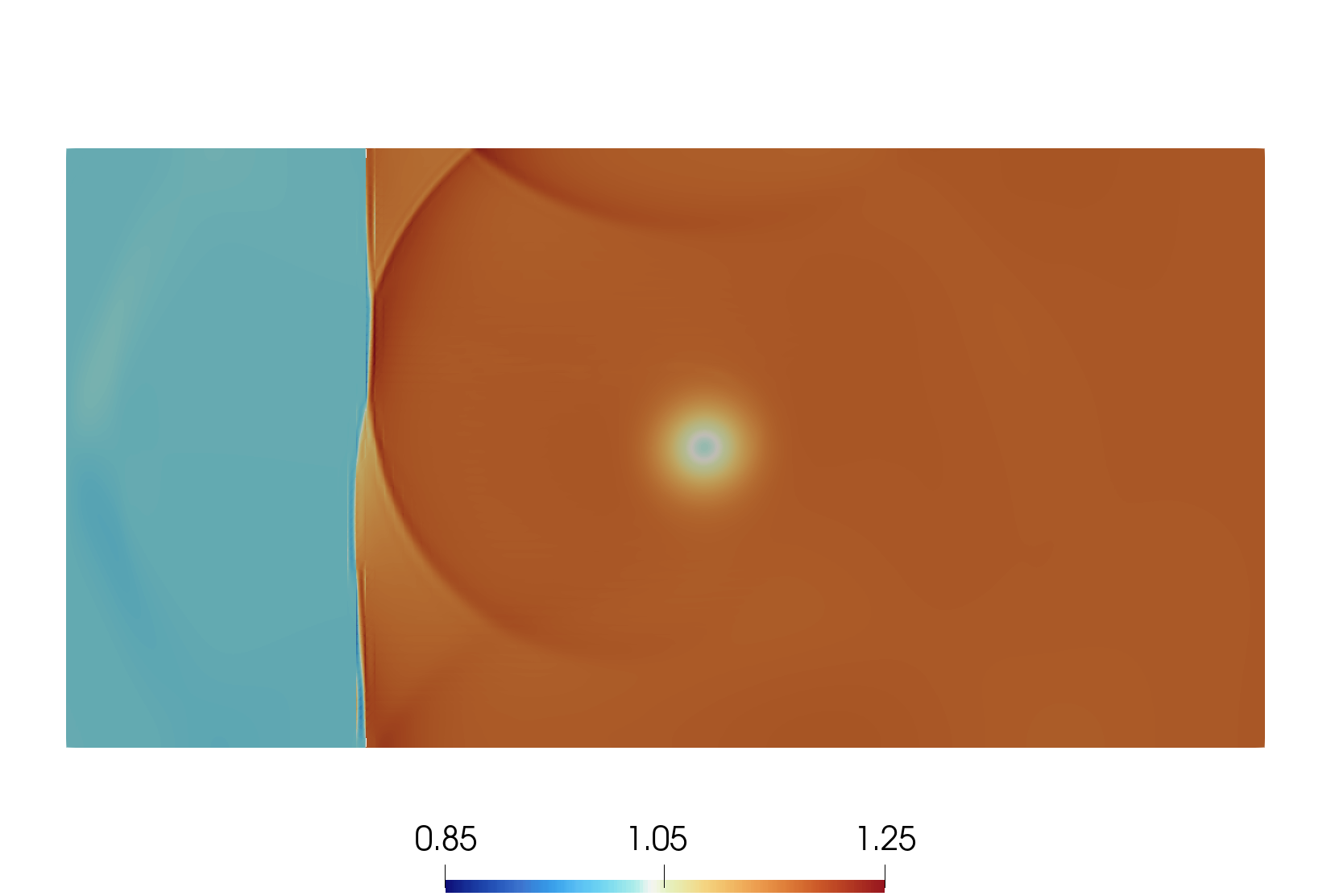}
    \caption{\centering Density at $t = 0.7$.}
    \label{fig: Schock Vortex interaction Density t =0.7}
\end{subfigure}
\hfill
\begin{subfigure}[t]{0.49\textwidth}
\centering
     \includegraphics[width = \textwidth]{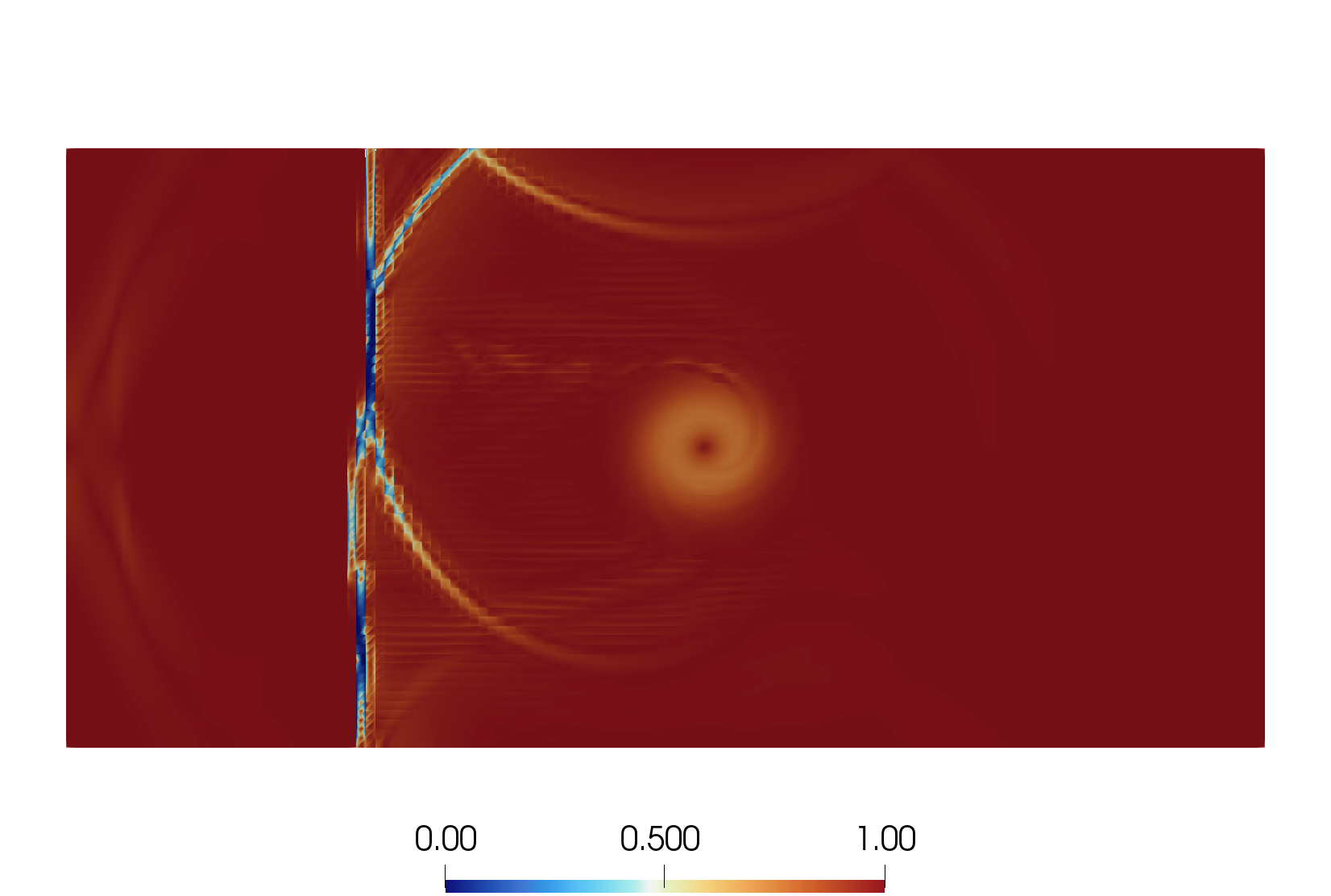}
     \caption{Schlieren plot at $t = 0.7$.}
    \label{fig: Schock Vortex interaction Schlieren t = 0.7}
\end{subfigure}
\caption{\reftwo{Example \ref{ex:Shock Vortex Interaction}}, density profiles and numerical Schlieren plots.}
\label{fig:Shock Vortex Interaction}
\end{figure}

\subsection{1D Density Wave}\label{ex: 1D Density Wave}

In the following example we compare the proposed ECAV method to a high order shock capturing method \cite{persson2006sub}. The shock capturing method utilizes the smoothness indicator on element $D^k$ introduced in \cite{hennemann2021provably} and motivated by \cite{klockner2011viscous} 
\begin{equation}\label{eq:smoothness indicator}
    S_k = \max{\left(\tilde{S}^N_k,\tilde{S}^{N-1}_k\right)}, 
\end{equation}
where $\tilde{S}^N_k$ is given by
\begin{equation}\label{eq:smoothness indicator arguments}
\tilde{S}^N_k = \frac{\LRp{u_h - \hat{u}_h,u_h-\hat{u}_h}_{D^k}}{\LRp{u_h,u_h}_{D^k}}, \quad u_h = \reftwo{\sum_{j=1}^{N+1}} u_j \psi_j, \quad \hat{u}_h = \reftwo{\sum_{j=1}^{N}} u_j \psi_j,
\end{equation}
and $u_h$ and $\hat{u}_h$ are the indicator variables expressed in terms of the orthogonal basis $\left\{\psi_j \right\}_{j=1}^N$. In \cite{hennemann2021provably} the indicator variables $u$ are chosen as the product of the density and pressure variables. \reftwo{The artificial viscosity coefficient used in this work is given by}

\begin{equation}\label{eq:ramp}
    \epsilon_k(s_k) = 
    \begin{cases}
        0 & \text{if} \quad  s_k < s_0 - \kappa\\
        \frac{\epsilon_0}{2} \LRp{1 + \sin{\frac{\pi(s_k - s_0)}{2 \kappa}}} & \text{if} \quad s_0 - \kappa \leq s_k \leq s_0 + \kappa\\
        \epsilon_0 & \text{if} \quad s_k > s_0 + \kappa,
    \end{cases}
\end{equation}
where $s_k = \log_{10}S_k$ and the parameters $\epsilon_0 \sim h/N$, $s_0 \sim \log_{10}(1/N^4)$, and $\kappa$ are chosen empirically so that a sharp but smooth shock profile is obtained. In the following numerical example, we use $s_0$ and $\kappa$ that are provided in \cite{PGSH2020}, where $s_0 + \kappa = -4\log_{10}{N}$ and $s_0 - \kappa = -11\log_{10}{N}$. We choose $\epsilon_0 = h/(2N)$; however, other values such as $\epsilon_0 = 3h/(2N)$ and $\epsilon_0 = 9h/(2N)$ resulted in similar behavior. 

We consider the 1D density wave problem with initial conditions:
\begin{equation}\label{eq:1D density wave initial conditions}
    \rho = 1.0 + 0.5 \exp{\left(- 10 \sin{\left(\pi x \right)^2} \right)}, \quad u = 0.1, \quad p = 10.0,
\end{equation}
with periodic boundary conditions. The solution to the density wave problem is approximated from $t = 0$ to a final time of $t = 25$ using the entropy correction artificial viscosity (ECAV) method and the shock capturing (SC) method. For both methods, we use degree $N = 5$ polynomials. 

\Cref{fig: 1D density wave epsilon} shows the time evolution of the maximum artificial viscosity coefficient and \Cref{fig: 1D density wave L2 error} shows the time evolution of the $L^2$ error for the shock capturing and entropy correction artificial viscosity methods using $K = 8$ and $K = 16$ elements. In \Cref{fig: 1D density wave profile}, the density profile at the final time for the ECAV and shock capturing method is compared to the exact solution at the final time. At these resolutions, we observe that the entropy correction artificial viscosity method generally results in a smaller maximum artificial viscosity and a smaller $L^2$ error than the shock capturing method. 

In \Cref{table:1D density errors}, we additionally compare $L^2$ errors of the ECAV and the DG method with no added viscosity at the final time $t = 25$. We observe that the error of the ECAV solution is larger than the $L^2$ norm of the difference between the ECAV solution and standard DG solution. The similarity of the two solutions is expected in this situation where we have a smooth solution and the ECAV method adds a minimal amount of artificial viscosity.

\begin{figure}[tbp]
\centering
\begin{subfigure}[t]{.49\textwidth}
\centering
    \includegraphics[width =\textwidth]{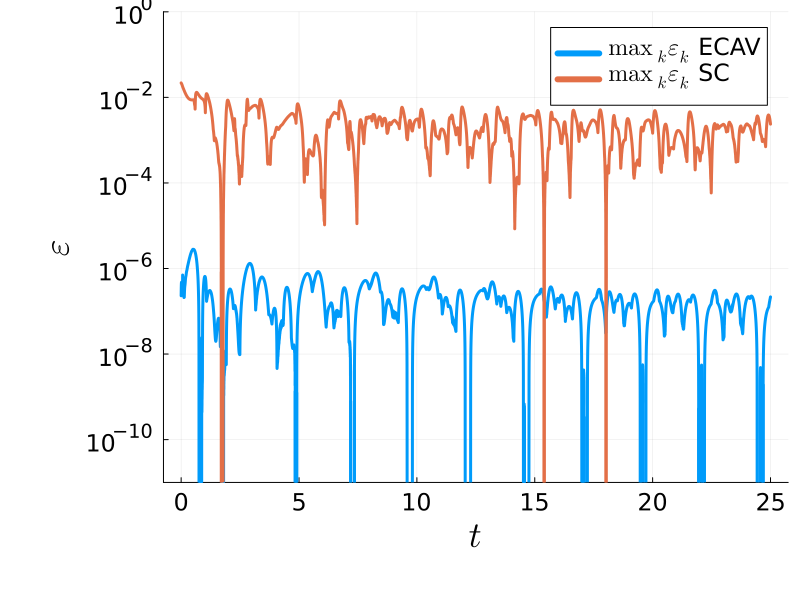}
    \caption{\centering $K = 8$}
    \label{fig: 1D density wave epsilon k = 8} 
\end{subfigure}
\hfill
\begin{subfigure}[t]{0.49\textwidth}
\centering
     \includegraphics[width = \textwidth]{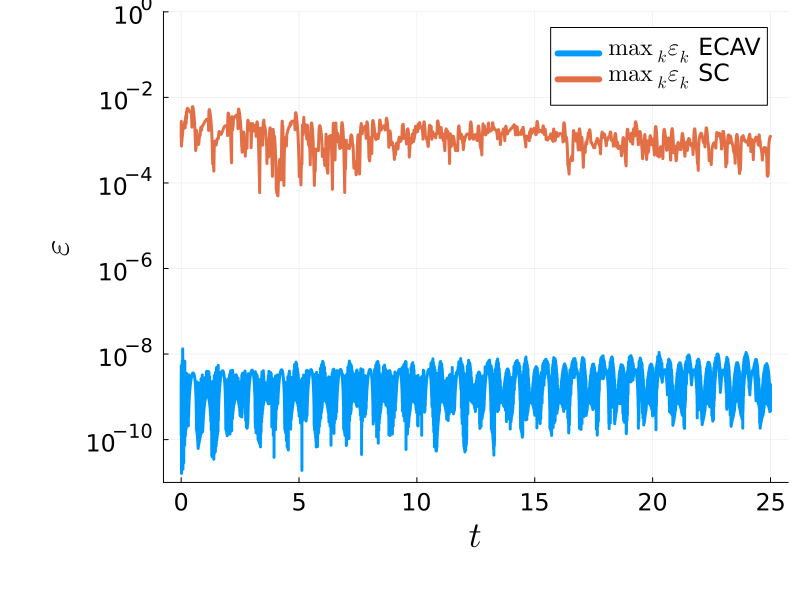}
     \caption{\centering $\max_k{\epsilon_k}$, $K = 16$}
    \label{fig: 1D density wave epsilon k = 16}
\end{subfigure}
\caption{Example \ref{ex: 1D Density Wave}, time evolution of the $\max_k{\epsilon_k}$ (semi-log plot) using shock capturing (labeled ``SC") and entropy correction artificial viscosity (labeled ``ECAV") for degree $N=5$ and both $K=8$ and $K = 16$ elements. }
\label{fig: 1D density wave epsilon}
\end{figure}


\begin{figure}[tbp]
\centering
\begin{subfigure}[t]{.49\textwidth}
\centering
    \includegraphics[width =\textwidth]{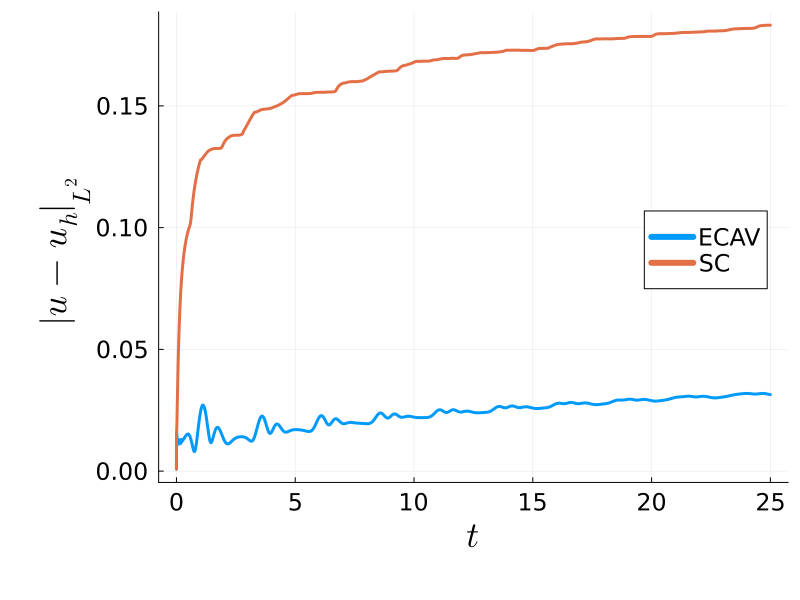}
    \caption{\centering $K = 8$}
    \label{fig: 1D density wave L2 error k = 8} 
\end{subfigure}
\hfill
\begin{subfigure}[t]{0.49\textwidth}
\centering
     \includegraphics[width = \textwidth]{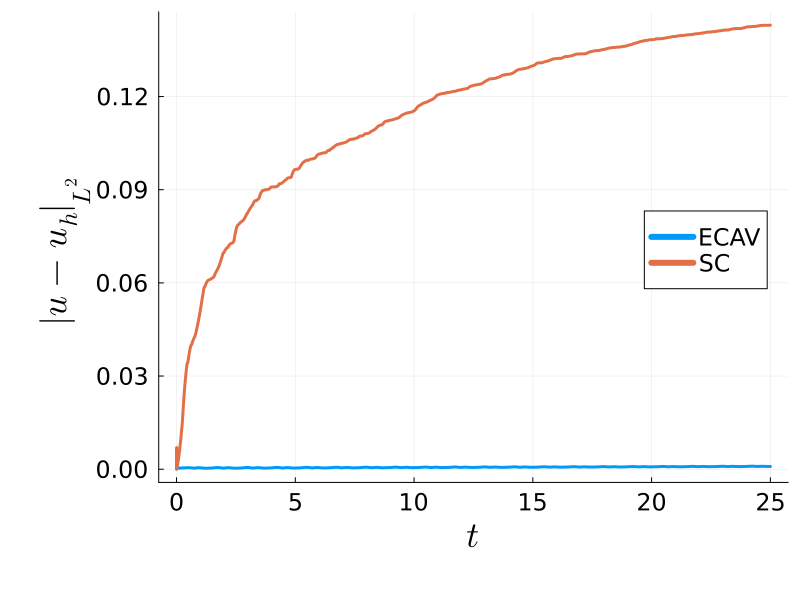}
     \caption{\centering \reftwo{$K = 16$}}
    \label{fig: 1D density wave L2 error k = 16}
\end{subfigure}
\caption{Example \ref{ex: 1D Density Wave}, time evolution of the $L^2$ error using shock capturing (labeled ``SC") and entropy correction artificial viscosity (labeled ``ECAV") for degree $N=5$ and both $K=8$ and $K = 16$ elements. }
\label{fig: 1D density wave L2 error}
\end{figure}

\begin{figure}[tbp]
\centering
\begin{subfigure}[t]{.49\textwidth}
\centering
    \includegraphics[width =\textwidth]{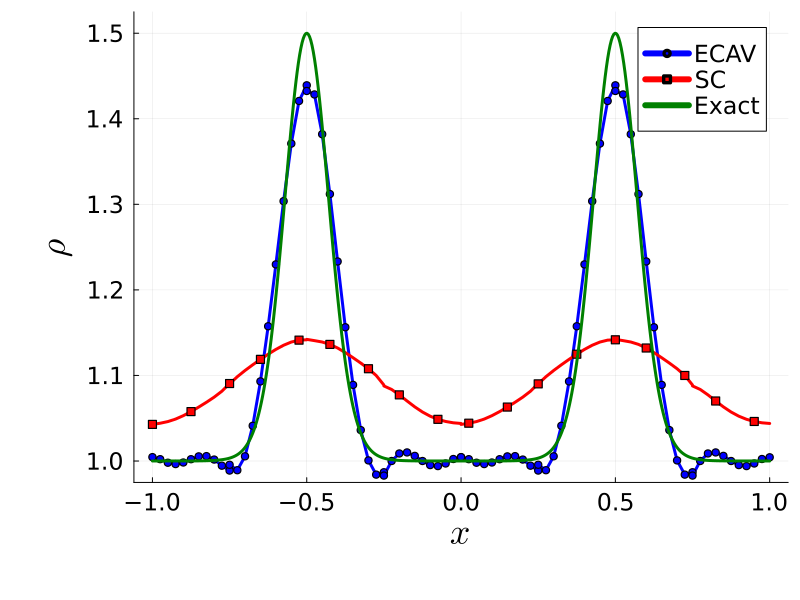}
    \caption{\centering $K = 8$}
    \label{fig: 1D density wave density profile k = 8.}
\end{subfigure}
\hfill
\begin{subfigure}[t]{0.49\textwidth}
\centering
     \includegraphics[width = \textwidth]{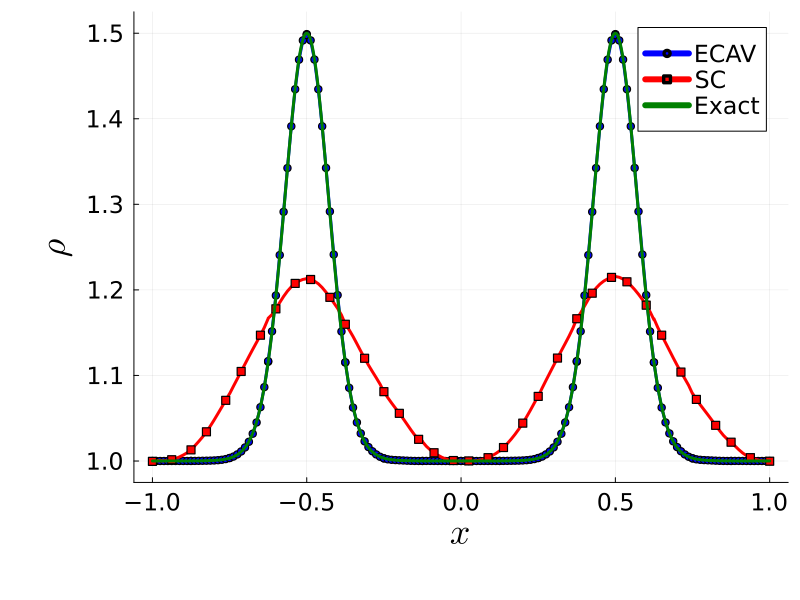}
     \caption{\centering $K = 16$}
    \label{fig: 1D density wave density profile k = 16}
\end{subfigure}
\caption{Example \ref{ex: 1D Density Wave}, Density profile at the final time $t = 25.0$ for the exact solution (labeled ``Exact"), the shock capturing solution (labeled ``SC"), and the entropy correction artificial viscosity solution (labeled ``ECAV") using $K = 8$ and $K = 16$ elements.}
\label{fig: 1D density wave profile}
\end{figure}
\begin{table}[tbp]

\centering
\begin{tabular}{|c|c|c|c|c|}
    \hline
     $K$ & $L^2$ error $\bm{u}_{\text{ECAV}}$ & $L^2$ error $\bm{u}_{\text{DG }}$ & $L^2 \text{ norm of }\bm{u}_{\text{ECAV}} - \bm{u}_{\text{DG }}$ \\
     \hline
     8 & $3.142 \times 10^{-2}$ & $3.140 \times 10^{-2}$ & $6.979 \times 10^{-5}$ \\
     \hline
     16 & $8.835 \times 10^{-4}$ & $8.831 \times 10^{-4}$ & $2.690 \times 10^{-5}$ \\
     \hline
     24 & $7.503\times 10^{-5}$ & $7.438\times 10^{-5}$ & $9.754\times 10^{-6}$ \\
     \hline
\end{tabular}
\caption{Example \ref{ex: 1D Density Wave}: A table comparing the errors of the ECAV solution and the DG (no added viscosity) solution  and the $L^2$ norm of the difference between the two methods at the final time $t = 25$.}
\label{table:1D density errors}
\end{table}

\subsection{Shu-Osher problem} \label{ex: Shu Osher}

This example problem is the Shu-Osher sine-shock interaction problem from \cite{shu2009high} with initial conditions
$$
(\rho, u, p) = \begin{cases}
    (3.857143, 2.629369,10.3333) & x < -4\\
    (1 + 0.2\sin{(5x)},0,1) & x \geq -4,
\end{cases}
$$
on the domain $[-5,5]$ and a final time of $t = 1.8$. We use $K = 100$ elements and polynomials of degree $N = 3$. \Cref{fig: Shu Osher dvPdv plot} is a plot of the time evolution of the quantity \reftwo{$\|\bm{\delta v}\|^2_{L^2(D^k)}/\|\Pi_N\bm{\delta v}\|^2_{L^2(D^k)}$} from \Cref{lemma: epsilon bound} and we see that it remains close to one with a modal LDG formulation. In \Cref{fig: Shu Osher Density Profile}, we can observe that the density profiles for local DG and BR-1 discretizations of ECAV are similar.

Interestingly, when using modal DG, the number of adaptive time steps when using LDG is 4218 compared to 7483 when using BR-1, although the maximum $\epsilon_k$ is often slightly larger for LDG. When using a nodal DG discretization (e.g., a discontinuous Galerkin spectral element discretization with $(N+1)$-point Legendre-Gauss-Lobatto quadrature, which results in a diagonal mass matrix and collocation-like scheme), the number of time steps is roughly the same, 4063 for LDG and 4116 for BR-1. 

In \Cref{fig: Shu Osher epsilon evolution}, we compare the time evolution of the maximum value of $\epsilon_k$ for BR-1 and LDG with both modal and nodal DG formulations. Because the maximum values of $\epsilon_k$ over time are similar for both LDG and BR-1, this suggests that the magnitude of the ECAV coefficient is not responsible for the increased number of time-steps under a BR-1 discretization. 

We explored other possible explanations for this behavior, such as  comparing $\|\bm{\sigma}\|_{L^2(\omega)}$ from \eqref{eq: LDG 2} or comparing the norm of the ODE right-hand side over time. In all examples, we did not observe significant differences between LDG and BR-1. Because the difference in the number of time-steps is observed only for the modal formulation, we conjecture that the projection of the entropy variables is a factor. However, other differences between LDG and BR-1 \cite{alhawwary2019study} could also be a factor.

\reftwo{In \Cref{fig: Shu Osher shock cap versus ECAV density}, we compare the density profile at the final time \(t = 1.8\) for the standard shock capturing method against the ECAV method implemented using both BR-1 and LDG discretizations. While the oscillations are smoothed using the shock capturing method discussed in Example \ref{ex: 1D Density Wave}, \Cref{fig: Shu Osher shock cap versus ECAV epsilon} shows that the ECAV methods generally have a smaller maximum $\epsilon_k$, although this does not result in the adaptive Runge-Kutta method using fewer time steps. We note that the shock capturing and ECAV methods used in \Cref{fig: Shu Osher shock cap versus ECAV density} and \Cref{fig: Shu Osher shock cap versus ECAV epsilon} use a nodal formulation as using a modal formulation for the shock capturing method caused the method to not be entropy stable and further parameter tuning would be needed to ensure entropy stability.}


\begin{figure}[tbp]
\centering
\begin{subfigure}[t]{.49\textwidth}
\centering
    \includegraphics[width =\textwidth]{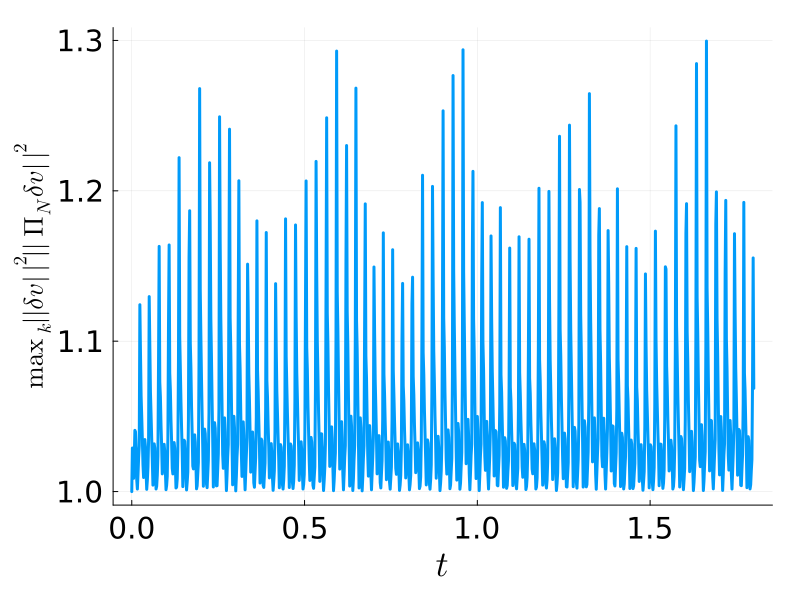}
    \caption{\centering $\|\bm{\delta v}\|^2_{L^2(D^k)}/\|\Pi_N\bm{\delta v}\|^2_{L^2(D^k)}$}
    \label{fig: Shu Osher dvPdv plot}
\end{subfigure}
\hfill
\begin{subfigure}[t]{0.49\textwidth}
\centering
     \includegraphics[width = \textwidth]{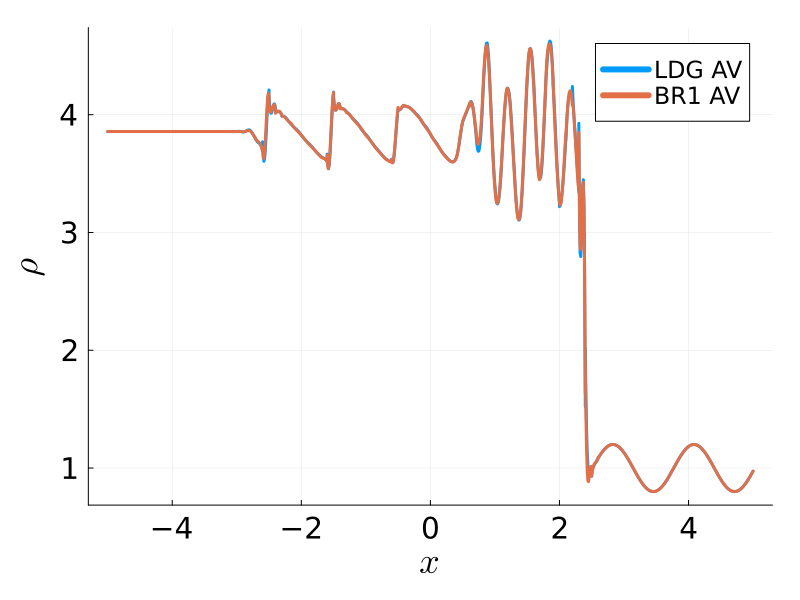}
     \caption{\centering Density profile}
    \label{fig: Shu Osher Density Profile}
\end{subfigure}
\caption{Example \ref{ex: Shu Osher}, Time evolution of $\|\bm{\delta v}\|^2_{L^2(D^k)}/\|\Pi_N\bm{\delta v}\|^2_{L^2(D^k)}$ from \Cref{lemma: epsilon bound} and density profile at the final time $t = 1.8$ for both the LDG and BR-1 \reftwo{ECAV} discretizations with a modal formulation.}
\label{fig: Shu Osher dvPdv and density plots}
\end{figure}

\begin{figure}[tbp]
\centering
\begin{subfigure}[t]{.49\textwidth}
\centering
    \includegraphics[width =\textwidth]{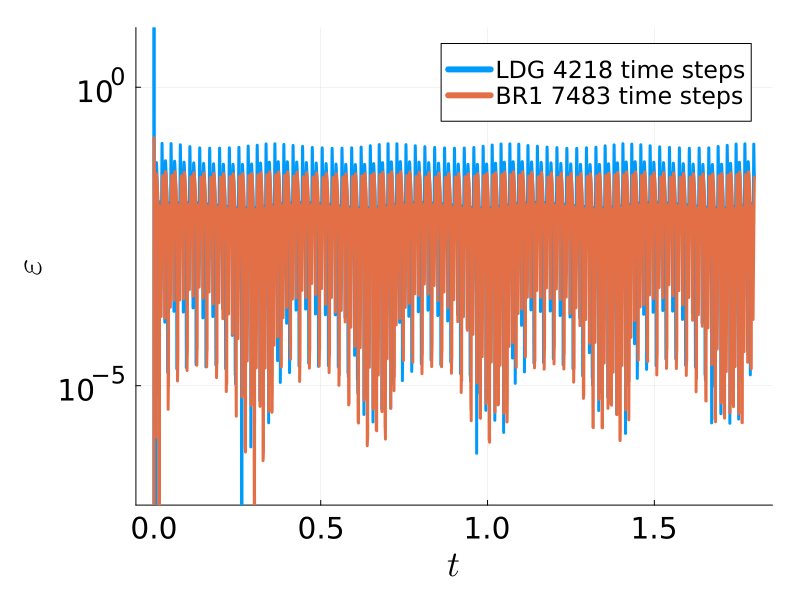}
    \caption{\centering Modal}
    \label{fig: Shu Osher Epsilon Modal}
\end{subfigure}
\hfill
\begin{subfigure}[t]{0.49\textwidth}
\centering
     \includegraphics[width = \textwidth]{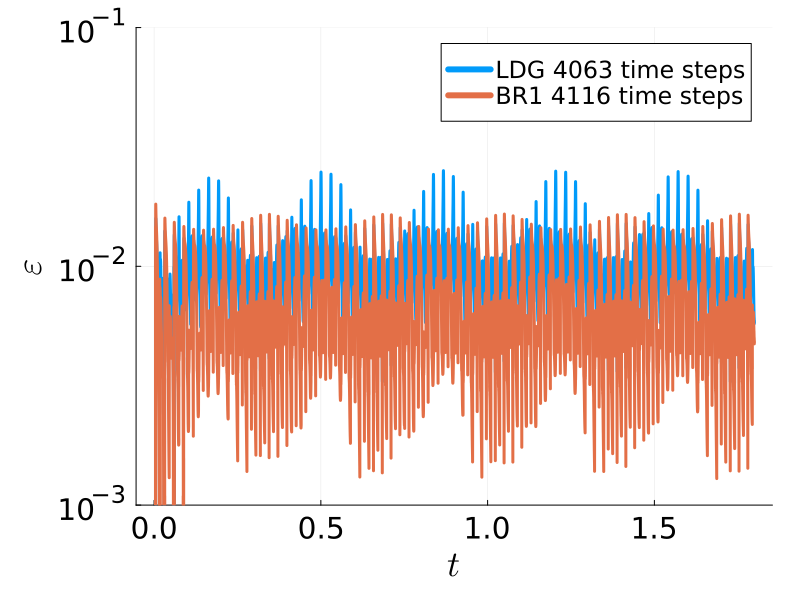}
     \caption{\centering Nodal}
    \label{fig: Shu Osher Density Nodal}
\end{subfigure}
\caption{Example \ref{ex: Shu Osher}, Time evolution of the  $\max_k \epsilon_k$ (semi-log plot) for modal and nodal formulations. }
\label{fig: Shu Osher epsilon evolution}
\end{figure}

\begin{figure}[tbp]
    \centering
    \begin{subfigure}[t]{.49\textwidth}
        \centering
        \includegraphics[width = \textwidth]{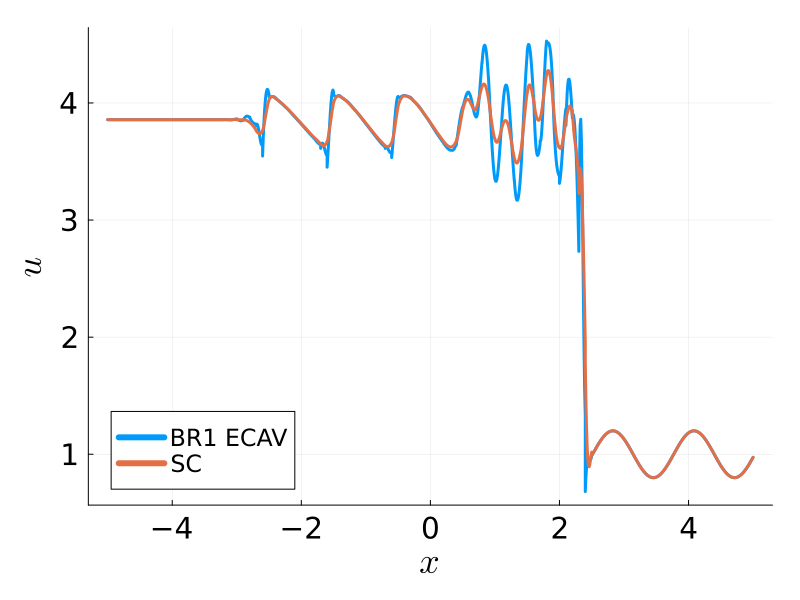}
        \caption{\centering Shock capturing versus ECAV with BR-1.}
        \label{fig: Shu Osher Density Shock Capture Compare BR1}
        \end{subfigure}
        \hfill
        \begin{subfigure}[t]{0.49\textwidth}
        \includegraphics[width = \textwidth]{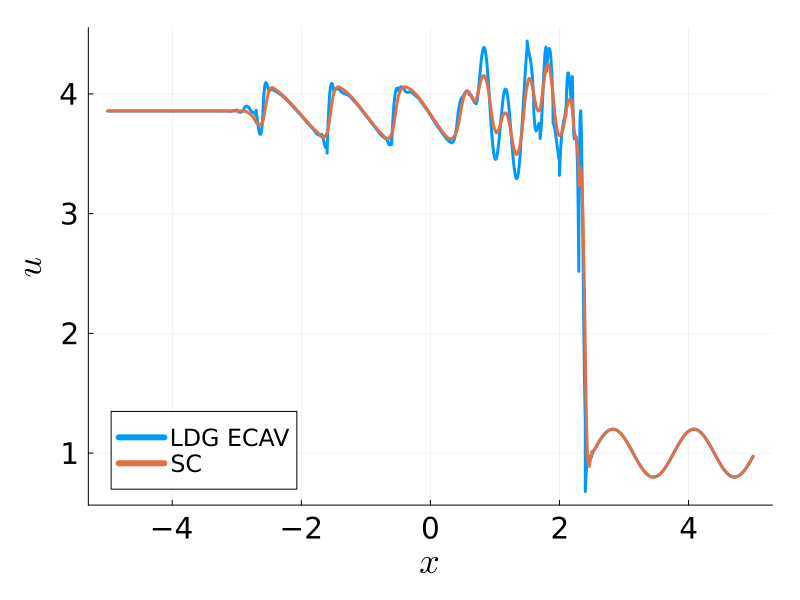}
        \caption{\centering Shock capturing versus ECAV with LDG.}
        \label{fig: Shu Osher Density Shock Capture Compare LDG}
    \end{subfigure}
    \caption{\reftwo{Example \ref{ex: Shu Osher}, Density profile at the final time $t = 1.8$ for shock capturing solution (labeled ``SC"), ECAV with BR-1 solution (labeled ``BR1 ECAV") and ECAV with LDG solution (labeled ``LDG ECAV").} }
    \label{fig: Shu Osher shock cap versus ECAV density}
\end{figure}

\begin{figure}[tbp]
    \centering
    \begin{subfigure}[t]{.49\textwidth}
    \centering
        \includegraphics[width = \textwidth]{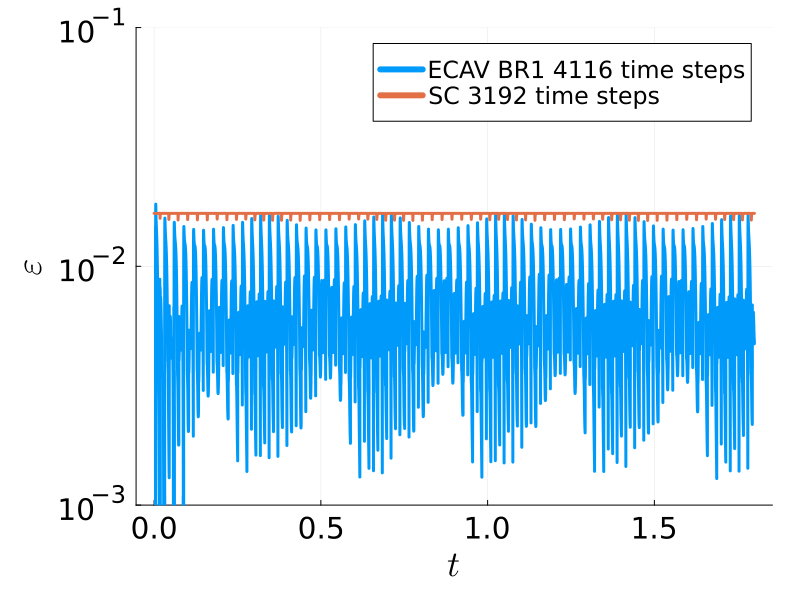}
        \caption{Shock capturing versus ECAV with BR-1.}
    \end{subfigure}
    \hfill
    \begin{subfigure}[t]{.49\textwidth}
        \centering
        \includegraphics[width = \textwidth]{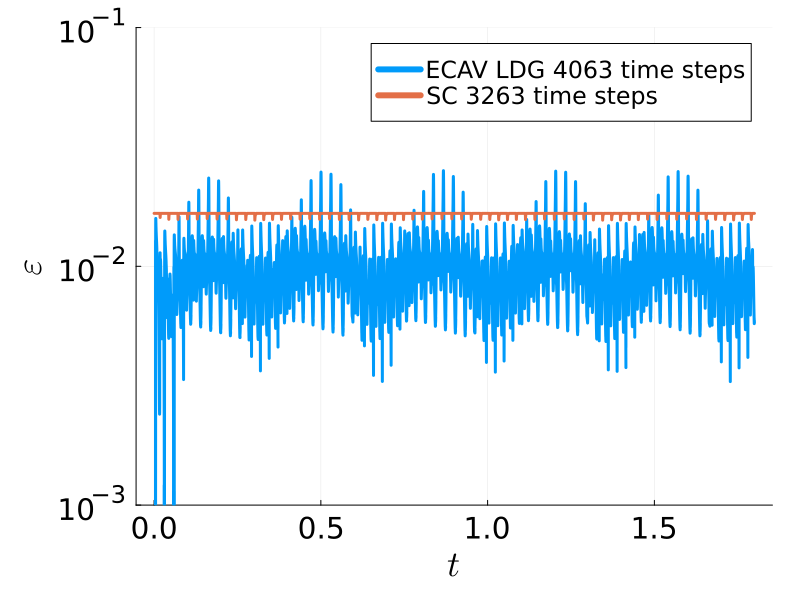}
        \caption{Shock capturing versus ECAV with LDG.}
    \end{subfigure}
    \caption{\reftwo{Example \ref{ex: Shu Osher}, Time evolution of the  $\max_k \epsilon_k$ (semi-log plot) for the shock capturing method and the ECAV with BR-1 and LDG methods.}}
    \label{fig: Shu Osher shock cap versus ECAV epsilon}
\end{figure}

\section{Conclusions}\label{sec:section_6}

In this paper, we analyze the behavior of entropy correction artificial viscosity (ECAV) \cite{chan2025artificial} under a local DG viscous discretization. We prove that, unlike BR-1 viscous discretizations, the norm of the DG gradients are bounded below, resulting in an $O(h)$ upper bound on the artificial viscosity coefficient. Additionally, we show that the resulting local DG viscous discretization satisfies the same entropy inequality derived in \cite{chan2025artificial} and is contact preserving.

Numerical experiments show that it is possible for the ECAV coefficient produced by the BR-1 viscous discretization to become arbitrarily large, resulting in a much smaller maximum stable time-step. We also compare against a more standard modal shock capturing method to illustrate that ECAV is minimally dissipative. Finally, we validate the high-order accuracy and entropy stability of the proposed method. 


The Julia implementation for the numerical experiments presented in this work is open-source and available on GitHub at \url{https://github.com/sqvanfleet/ECAV\_LDG}. The specific version of the code used to generate the results in this paper is permanently archived on Zenodo \cite{zenodo_ECAV_LDG}.

\section{Acknowledgements}

The authors thank Keegan Kirk and Hendrik Ranocha for helpful discussions. Jesse Chan gratefully acknowledges support from National Science Foundation (NSF) under award DMS-1943186. Samuel Van Fleet acknowledges support from the NSF under award DMS-223148. 

\section{Declaration of generative AI in the writing process}

During the preparation of this work, the authors used AI tools Cursor's Composer 2.5, ChatGPT 5.6 Sol, and Google Gemini to summarize information, detect typographical errors, and assist in the generation of figures. The authors reviewed and edited the content as needed and takes full responsibility for the content of the publication.

\bibliographystyle{plain}
\bibliography{reference.bib}

\end{document}